\documentclass[12pt]{amsart}
\usepackage{latexsym}
\usepackage{amssymb}
\usepackage{amscd}
\usepackage{mathrsfs}
\usepackage{rotating}
\usepackage{tabu}
\usepackage[enableskew]{youngtab} 

\newtheorem{thm}{Theorem}[subsection]
\newtheorem{lem}[thm]{Lemma}
\newtheorem{cor}[thm]{Corollary}
\newtheorem{prop}[thm]{Proposition}
\newtheorem{exa}[thm]{Example}
\newtheorem{rem}[thm]{Remark}
\newtheorem{Def}[thm]{Definition}

\newcommand\la{\lambda}

\newcommand\BW{\mathrm W}

\newcommand\Bla{\boldsymbol\lambda}

\newcommand\Bmu{\boldsymbol\mu}

\newcommand\BBw{\boldsymbol{w}}
\newcommand{\isom}{\,\raise2pt\hbox{$\underrightarrow{\sim}$}\,}
\numberwithin{equation}{subsection}

\begin{document}
\setlength{\baselineskip}{4.9mm}
\setlength{\abovedisplayskip}{4.5mm}
\setlength{\belowdisplayskip}{4.5mm}
\renewcommand{\theenumi}{\roman{enumi}}
\renewcommand{\labelenumi}{(\theenumi)}
\renewcommand{\thefootnote}{\fnsymbol{footnote}}
\renewcommand{\thefootnote}{\fnsymbol{footnote}}
\allowdisplaybreaks[2]
\parindent=20pt

\pagestyle{myheadings}
\medskip
\begin{center}
{\bf Fermionic formula for double Kostka polynomials }
\end{center}
\par\bigskip
\begin{center}
Shiyuan Liu
\\
\vspace{0.5cm}
\end{center}

\title{}

\begin{abstract}
The $X=M$ conjecture asserts that the $1D$ sum and the fermionic formula coincide up to some constant power. In the case of type $A,$ both the $1D$ sum and the fermionic formula are  closely related to  Kostka polynomials. Double Kostka polynomials $K_{\Bla,\Bmu}(t),$ indexed by two double partitions $\Bla,\Bmu,$ are polynomials in $t$ introduced as a generalization of  Kostka polynomials. In the present paper, we consider $K_{\Bla,\Bmu}(t)$ in the special case
where $\Bmu=(-,\mu'').$ We formulate  a $1D$ sum  and  a fermionic formula  for  $K_{\Bla,\Bmu}(t),$ as a generalization of the case of ordinary Kostka polynomials. Then we prove an analogue of the $X=M$ conjecture.

\bigskip

\end{abstract}

\maketitle
\markboth{}{}
\pagestyle{myheadings}


\begin{center}
{\sc Introduction}
\end{center}
\par\medskip

\subsection{}
A certain polynomial in $t,$ called the "fermionic formula", $M$ was introduced in [HKOTY]. It can be viewed as a $t$-analogue of the multiplicity in certain representation of an affine quantum group which first appeared in Kirillov and Reshetikhin [KR2]. In [HKOTT],  a polynomial called the "$1D$ (one-dimensional) sum" $X$ was introduced by using the terminology of crystals. $1D$ sums and fermionic formulas are both defined for tensor products of KR (Kirillov-Reshetikhin) modules [KR2]. The so-called $X=M$ conjecture [HKOTT] asserts that $1D$ sum and the fermionic formula coincide up to some constant power of $t$.

\subsection{}
In the case of type $A^{(1)}_n,$ the $X=M$ conjecture is described as follows. Let $\mathfrak{g}$ be the affine algebra of type $A_n^{(1)}$ with index $I=\{0,1,\ldots,n\},$  and $\mathfrak{g}_{0}$ the corresponding finite dimensional simple Lie algebra. Let $U_q(\mathfrak{g})$  be the corresponding quantized universal enveloping algebra. Set $I_0=I-\{0\}$ and $\mathscr{H}=I_0\times\mathbb{Z}_{>0}.$ Let $\mathscr{P}_N$ be the set of partitions of size $N.$  In this paper, we consider $B(\mu)=B^{1,\mu_r}\otimes \cdots B^{1,\mu_1}$ for KR-crystals $B^{1,\mu_i},$ where $\mu=(\mu_1,\ldots,\mu_r)\in\mathscr{P}_N.$ For $b\in B(\mu),$ a non-negative integer $\mathrm{E}(b)$ attached to $b,$ called the energy, was introduced in [NY].
Let $\lambda$ be a dominant weight with respect to $\mathfrak{g}_0.$
Let $P(B(\mu),\lambda)$ be the subset of $B(\mu)$ consisting of $U_{q}(\mathfrak{g}_0)$-maximal elements with respect to the highest weight $\lambda.$ Under these notations, the $1D$ sum is given by
\begin{equation}\label{0.2.1}
X(\mu,\lambda;t)=\sum_{b\in P(B(\mu),\lambda)}t^{\mathrm{E}(b)}.
\end{equation}
 For $p,m\in\mathbb{Z}_{\geq 0},$ define
$$\left[ { p+m \atop m} \right]
_t=\frac{(t^{p+1}-1)(t^{p+2}-1)\cdots (t^{p+m}-1)}{(t-1)(t^2-1)\cdots(t^m-1)}.$$  On the other hand, the fermionic formula with respect to $\lambda$ is given by
\begin{equation}\label{0.2.2}
M(\mu,\lambda;t) =\sum_{\{m\}}t^{\mathrm{cc}(\{m\})}\prod_{(a,i)\in\mathscr{H}}\left[ { p^{(a)}_i+m_i^{(a)} \atop m_i^{(a)} } \right]_t,
\end{equation}
where $\mathrm{cc}\{m\}$ and $p_i^{(a)}$ are defined by
\begin{align*}
&\mathrm{cc}(\{m\}) =\frac{1}{2}\sum\limits_{a,b\in I_0,i,j\geq 1}(\alpha_a,\alpha_b)\min(i,j)m_{i}^{(a)}m_{j}^{(b)},\\
&p_{i}^{(a)}=\delta_{a,1}\sum_{j=1}^r\min(i,\mu_j)-\sum_{(b,j)\in\mathscr{H}}(\alpha_a,\alpha_b)\min(i,j)m_j^{(b)}.
\end{align*}
Here $(\cdot,\cdot)$ is the normalized invariant form on the weight lattice of $\mathfrak{g}_0.$
The sum $\sum_{\{m\}}$ is taken over all the $\{m^{(b)}_j\in\mathbb{Z}_{\geq 0}\mid (b,j)\in\mathscr{H}\}$ such that $p^{(a)}_i\geq 0$ for any $(a,i)\in\mathscr{H}$ and $N\overline{\Lambda}_1-\sum_{(a,i)\in\mathscr{H}}im_i^{(a)}\alpha_a=\lambda,$ where $\overline{\Lambda}_1$ is the $1$st fundamental weight with respect to $\mathfrak{g}_{0}.$

\medskip
In our case, the $X=M$ conjecture can be proved  by using the combinatorics of crystals and rigged configurations, see, e.g., [Sc2, KSS].
Rigged configurations are combinatorial objects which first appeared in the Bethe's paper on the study of the Bethe Ansatz,  and were generalized by Kerov, Kirillov and Reshetikhin [KR1,KKR]. In the view point of rigged configurations, as in [Sc2, (4.6)], we can rewrite (\ref{0.2.2}) as
\begin{equation}\label{0.2.3}
M(\mu,\lambda;t)=\sum_{(\nu,J)\in\mathrm{RC}(\mu,\lambda)}t^{\mathrm{cc}(\nu,J)},
\end{equation} where $\mathrm{RC}(\mu,\lambda)$ is the set of valid highest weight $L(\mu)$-rigged configurations of highest $\lambda$ and $\mathrm{cc}(\nu,J)$ is a non-negative integer, called the cocharge of $(\nu,J).$ See section 4 for detailed definitions.

\subsection{}
The Kostka polynomials $K_{\lambda,\mu}(t),$ indexed by  two partitions $\lambda$ and $\mu,$ play an important role in the combinatorial theory, the representation theory and the mathematical physics, see [M, III] for details.
For a tableau $T$ of partition weight, a non-negative integer $\mathrm{c}(T),$ called the charge of $T,$ see, e.g., [B], [M]. Let $\mathrm{Tab}(\lambda,\mu)$ be the set of tableaux of shape $\lambda$ and weight $\mu.$ Lascoux and Sch\"{u}tzenberger proved that the  Kostka polynomial $K_{\lambda,\mu}(t)$ can be expressed as the generating functions of $\mathrm{Tab}(\lambda,\mu)$ with respect to charge statistic ([M, III, 6.5])
\begin{equation*}
K_{\lambda,\mu}(t)=\sum_{T\in\mathrm{Tab}(\lambda,\mu)}t^{\mathrm{c}(T)}.
\end{equation*}

Let $\mu\in\mathscr{P}_N ,$ and $\lambda$  a highest weight of $B(\mu).$ It is well-known that $\lambda$ can be identified with a partition of size $N.$  In [NY], Nakayashiki and Yamada constructed a bijection $\Psi_\lambda:\mathrm{Tab}(\lambda,\mu)\isom P(B(\mu),\lambda).$ They proved that $\Psi_\lambda$ preserves the charge and energy, namely, $\mathrm{c}(T)=\mathrm{E}(\Psi_\lambda(T))$ for $T\in\mathrm{Tab}(\lambda,\mu).$ By using this bijection, they expressed the Kostka polynomial $K_{\lambda,\mu}(t)$ as
\begin{equation}\label{0.3.1}
K_{\lambda,\mu}(t)=\sum_{b\in P(B(\mu),\lambda)}t^{\mathrm{E}(b)}.
\end{equation}

In [KR1], Kirillov and Reshetikhin constructed a bijection $\Pi_\lambda:\mathrm{Tab}(\lambda,\mu)\rightarrow\mathrm{RC}(\mu,\lambda)$ and proved that
\begin{equation}
K_{\lambda,\mu}(t)=t^{n(\mu)}M(\mu,\lambda;t^{-1}),
\end{equation}
where $n(\mu)=\sum_{i=1}^{r}(i-1)\mu_i$ for $\mu=(\mu_1,\ldots,\mu_r).$ In \cite{Sc}, Schilling defined the  Kashiwara operators $e_i,f_i$ ($i\in I_0$) on the set of rigged configurations and extended the definition of cocharge to any rigged configuration.  Let $\mathrm{RC}(\mu)$ be the set of $L(\mu)$-valid rigged configurations generated from all the valid highest weight $L(\mu)$-rigged configurations by the application of the Kashiwara operators. It is known (essentially due to \cite{Sc}) that $\mathrm{RC}(\mu)$ has a crystal structure containing $\mathrm{RC}(\mu,\lambda)$ as the subset of maximal elements of highest weight $\lambda.$    In [KSS], a bijection $\Phi:B(\mu)\rightarrow \mathrm{RC}(\mu)$, called the rigged configuration bijection, was established. It can be viewed as an extension of the bijection $\Phi_\lambda=\Pi_\lambda\circ\Psi_\lambda^{-1}:P(B(\mu),\lambda)\isom \mathrm{RC}(\mu,\lambda).$  In [Sa] and [DS], it was proved that the rigged configuration bijection $\Phi$ commutes with the Kashiwara operators, which implies that $\Phi$ is a crystal isomorphism.

\medskip

Summing up the above discussion, we obtain the following commutative diagram
\begin{equation}\label{0.2.6}\begin{CD}
         @.                             B(\mu)                        @>\Phi>>             \mathrm{RC}(\mu)\\
@.                                       @AAA                                           @AAA\\
\mathrm{Tab}(\lambda,\mu) @>\Psi_\lambda>>    P(B(\mu),\lambda)   @>\Phi_\lambda>>      \mathrm{RC}(\mu,\lambda),
\end{CD}
\end{equation}
where the vertical maps are natural inclusions.

\subsection{}  Let $\mathscr{P}_{N,2}$ be the set of pair of  partitions $(\lambda',\lambda'')$ such that the sum of the sizes $\lambda'$ and  $\lambda''$ equals $N.$ We call the elements in $\mathscr{P}_{N,2}$ double partitions.  In [S1, S2], Kostka polynomials $K_{\Bla,\Bmu}(t),$ indexed by two double partitions $\Bla,\Bmu,$ were introduced as a generalization of ordinary Kostka polynomials.  as introduced in [LS]. In this paper, we call them double Kostka polynomials.
In [LS], a set $\mathrm{Tab}(\Bla,\mu)$ of tableaux of shape $\Bla$ and weight $\mu$ was introduced. For  $\Bla,\Bmu\in\mathscr{P}_{N,2}$ with $\Bmu=(-,\mu''),$ a Lascoux-Sch\"{u}tzenberger type formula for  double Kostka polynomials $K_{\Bla,\Bmu}(t)$ was given in [LS, Theorem 3.12]
\begin{equation}
K_{\Bla,\Bmu}(t)=t^{|\lambda'|}\sum_{T\in\mathrm{Tab}(\Bla,\mu'')}t^{2\mathrm{c}(T)},
\end{equation} where $\mathrm{c}(T)$ is a certain charge of $T.$

\subsection{}
In this paper, we consider the double Kostka polynomials $K_{\Bla,\Bmu}(t)$ in the special case where $\Bmu=(-,\mu'').$ As a generalization of the case of Kostka polynomials, we give the $1D$ sum expression and the fermionic formula for $K_{\Bla,\Bmu}(t),$ proving an analogue of the $X=M$ conjecture.
Our results are summarized as follows.

\par\medskip\noindent\textbf{Result I.}

In Proposition \ref{thm 3.3.3} and Theorem \ref{cor 3.3.4}, we construct a crystal $\mathrm{W}(\mu)$ and an isomorphism of crystals $\Psi:\mathrm{W}(\mu)\rightarrow B(\mu).$  For any $\lambda\in\mathscr{P}_{N},$ $\mathrm{Tab}(\lambda,\mu)$ can be naturally embedded into $\mathrm{W}(\mu)$ and $\Psi$ can be viewed as an extension of the map $\Psi_\lambda.$ Namely,  the diagram (\ref{0.2.6}) can be completed as follows
\begin{equation}\begin{CD}
\mathrm{W}(\mu)           @>\Psi>>                               B(\mu)                        @>\Phi>>             \mathrm{RC}(\mu)\\
@AAA                                          @AAA                                           @AAA\\
\mathrm{Tab}(\lambda,\mu) @>\Psi_\lambda>>    P(B(\mu),\lambda)   @>\Phi_\lambda>>      \mathrm{RC}(\mu,\lambda).
\end{CD}
\end{equation}
We also define a charge function $\mathrm{c}:\mathrm{W}(\mu)\rightarrow \mathbb{Z}_{\geq 0},$ which is an extension of the charge on $\mathrm{Tab}(\lambda,\mu).$ In Theorem \ref{Thm 3.5.2} and Corollary 3.9.3, we  prove that $\Psi$
preserves the charge function and energy function.

\par\medskip\noindent\textbf{Result II.}

Recall that for $\Bla,\Bmu\in \mathscr{P}_{N,2}$ with $\Bmu=(-,\mu''),$ the double Kostka polynomial $K_{\Bla,\Bmu}(t)$ is described by the set $\mathrm{Tab}(\Bla,\mu'')$ and a certain charge on it, as in 0.4. In the present paper, we show, for $\Bla\in\mathscr{P}_{N,2},$ $\mu\in\mathscr{P}_{N},$ that  $\mathrm{Tab}(\Bla,\mu)$ is naturally embedded in $\mathrm{W}(\mu).$  We modify the definition of the charge on $\mathrm{Tab}(\Bla,\mu)$ in [LS] such that it coincides with the restriction of the charge on $\mathrm{W}(\mu).$ In Proposition 3.8.1, we construct a suitable subset $P(B(\mu),\Bla)$ of $B(\mu)$ and prove that the restriction of $\Psi$ on $\mathrm{Tab}(\Bla,\mu)$ gives a bijection $\mathrm{Tab}(\Bla,\mu)\isom P(B(\mu),\Bla).$ We define a set $\mathrm{RC}(\mu,\Bla)$ as the image of $P(B(\mu),\Bla)$ under the isomorphism $\Phi.$  In Corollary 4.9.3, we determine $\mathrm{RC}(\mu,\Bla)$ explicitly. Summing up the above discussion, we obtain the following commutative diagram
\begin{equation}\label{0.4.2}\begin{CD}
\mathrm{W}(\mu)           @>\Psi>>                               B(\mu)                        @>\Phi>>             \mathrm{RC}(\mu)\\
@AAA                                          @AAA                                           @AAA\\
\mathrm{Tab}(\Bla,\mu) @>\Psi_{\Bla}>>    P(B(\mu),\Bla)   @>\Phi_{\Bla}>>      \mathrm{RC}(\mu,\Bla),
\end{CD}
\end{equation}
 where the vertical maps are inclusion maps, and $\Psi_{\Bla}$ (resp. $\Phi_{\Bla}$) is the restriction of $\Psi$ (resp. $\Phi$) which gives a bijection.

 We remark that in contrast to the case in 0.2, the set $P(B(\mu),\Bla)$ no longer consists of maximal elements and $\mathrm{\mu,\Bla}$ does not consist of restricted rigged configurations in the sense of Schilling [Sc1].

\par\medskip\noindent\textbf{Result III.}

Let $\Bla,\Bmu\in\mathscr{P}_{N,2}$ as in 0.3. By using bijection $\Psi_{\Bla}$ given in (\ref{0.4.2}) and the fact that $\Psi_{\Bla}$ preserves the charge and the energy function, we prove the $1D$ sum expression for  double Kostka polynomials
\begin{equation}
K_{\Bla,\Bmu}(t)=t^{|\lambda'|}\sum_{b\in P(B(\mu''),\Bla)}t^{2\mathrm{E}(b)},
\end{equation} which gives an analogue of (\ref{0.3.1}). Hence, $K_{\Bla,\Bmu}(t)$ can be viewed as the generating function over $P(B(\mu''),\Bla)$ with respect to the twice energy function.
In Definition 4.11.2, we define a fermionic formula $M(\Bla,\mu;t).$ In Theorem 4.11.3, we prove that
\begin{equation}
M(\Bla,\mu'';t^2)=t^{2n(\mu'')+|\lambda'|}K_{\Bla,\Bmu}(t^{-1}).
\end{equation} Hence, $M(\Bla,\mu'';t^2)$ can be viewed as the fermionic formula for  double Kostka polynomials $K_{\Bla,\Bmu}(t).$  In Remark 4.11.4, we define a $1D$ sum
\begin{equation}
X(\mu,\Bla;t)=\sum_{b\in P(B(\mu),\Bla)}t^{\mathrm{E}(b)}.
\end{equation} By summing up the above discussion, we obtain
\begin{equation}
X(\mu,\Bla;t)=t^{n(\mu)}M(\mu,\Bla;t^{-1}),
\end{equation} which give an analogue of the $X=M$ conjecture.

\bigskip

\par\noindent \textbf{Acknowledgements}

\medskip

The author is indebted to Professor Toshiaki Shoji, who gave a lot of suggestions for completing the paper.   He would like to thank Professor Masato Okado  for his valuable advices during the preparation of this paper.  Professor Okado gave a series of lectures on $X = M$ conjecture at Tongji University in September 2014,  which is the origin of this research.  Thanks are also due to Yanmin Pu  for the careful reading of the draft of the paper.

\par\smallskip
\section{Combinatorial preliminaries}
 In this section, we introduce some combinatorial notions and tools used in the present paper.
 \subsection{}

Let $\lambda=(\lambda_1,\lambda_2,\ldots,\lambda_k)$ be a partition, namely $\lambda_1\geq\lambda_2\geq\cdots\geq\lambda_k>0,$ $\lambda_i\in\mathbb{Z}.$  Let $|\lambda|=\sum^{k}_{i=1}\lambda_i$ be the size of $\lambda$ and $l(\lambda)=k$ be its length. We denote by $\mathscr{P}_N$  the set of partitions of size $N$. We usually identify $\lambda$ with its Young diagram, which is a collection of boxes, arranged in left-justified row, with $\lambda_i$ boxes in the $i$-th row.

Let $\lambda=(\lambda_1,\ldots,\lambda_s)$ and $\rho=(\rho_1,\ldots,\rho_t)$ be two partitions such that $s\geq t$ and $\lambda_i\geq \rho_i$ ($1\leq i\leq t$). In this case, we write $\lambda\supset\rho$ and define the skew shape $\lambda-\rho$ to be the diagram obtained by removing the Young diagram $\rho$ from the Young diagram $\lambda.$

A tableau is a filling of the boxes of a skew shape with positive integers, weakly increasing in rows and strictly increasing in columns. Let $T$ be a tableau of shape $\lambda-\rho.$ Set $|T|=|\lambda|-|\rho|.$ We denote by $L(T)$ the set of letters appearing in $T.$

  We write words as a sequence of letters (positive integers, with our conventions). For a word $w=a_{m}a_{m-1}\cdots a_{1}$ of length $m,$ let
$\mu_i$ be the number of occurrences of the letter $i$ in the word $w.$ Define the weight of $w$ to be
$$\mathrm{wt}(w)=(\mu _{1},\mu
_{2},\ldots ).$$ $w$ is called a standard word if $\mathrm{wt}(w)=(1^m).$ Let $w_1=a_n\cdots a_1$ and $w_2=b_m\cdots b_1$ be two words, write $$w_1*w_2=a_n\cdots a_1b_m\cdots b_1.$$ It is clear that $\mathrm{wt}(w_1*w_2)=\mathrm{wt}(w_1)+\mathrm{wt}(w_2).$

To each tableau $T$, we associate a word $w(T)=a_{|T|}a_{|T|-1}\cdots a_{1}$, called the word of $T$, where the sequence $a_{1}, a_{2}, \ldots, a_{|T|}$ is derived by reading the letters in $%
T$ from right to left in successive rows, starting with the top row. We say a word $w$ is compatible with $\lambda-\rho $, if there exists some tableau $T$ of shape $\lambda-\rho$ such that $w(T)=w.$ Note that such a tableau is unique if $\lambda-\rho$ is fixed. A word $w=a_m\cdots a_1$ is called a row if $w$ is compatible with the partition $(m)$, namely
$a_m\leq a_{m-1}\leq\cdots \leq a_1.$ Similarly, a word $w=b_m\cdots b_1a_m\cdots a_1$ of even letters is called a double-row if it is compatible with the partition $(m^2)$, namely
$a_m\leq a_{m-1}\leq\cdots \leq a_1$, $b_m\leq b_{m-1}\leq\cdots \leq b_1$ and $a_i<b_i$ for each $1\leq i\leq m.$ Define the weight of a tableau to be the weight of its word.

We say a tableau $T$ is standard if $w(T)$ is a standard word. For a partition $\lambda$, let $\mathrm{Tab}(\lambda)$ denote the set of tableaux of shape $\lambda.$ Denote by $\mathrm{ST}(\lambda)$ the subset of $\mathrm{Tab}(\lambda)$ consisting of standard tableaux.

\subsection{} Following [B, 2.3], we consider the column bumping operation on tableaux. For a letter $a$ and a tableau $T$ of partition shape, let  $a\rightarrow T$  denote the column bumping procedure. The resulting tableau is denoted by $[a\rightarrow T].$  We write $R(a\rightarrow T)=r$ if the new box appears in the $r$-th row of $[a\rightarrow T].$

Let $w=a_{m}a_{m-1}\cdots a_{1}$ be a word and $T$   a tableau of partition shape. Set
$$[w\rightarrow T]=[a_m\rightarrow [a_{m-1}\rightarrow \cdots \rightarrow[a_1\rightarrow T]\cdots]].$$

\begin{lem}\label{lem 1.2.1}
Let $T$ be a tableau of partition shape, and let $u$, $v$ be two
letters. Set $R_{u}=R(u\rightarrow T)$ and $R_{v}=R(v%
\rightarrow [u\rightarrow T])$. If $v\leq u$, then $R_{v}\leq R_{u}$.
\end{lem}
\begin{proof}
We use induction on the number of the columns of $T$. If $T$ has no column, namely, $T=\emptyset,$ it is easy to check $R_{v}=R_{u}=1$.

Assume that $T$ has $r\geq 1$ columns. If $u$ bumps no letter in the first column of $T$, then $R_{u}=n+1$, where $n$ is the length of the first
column.
Note that $v$ must bump some letter in the first column of  $[u\rightarrow T]$ since $v\leq u$, hence $R_{v}\leq n+1=R_{u}$.
If $u$ bumps a letter $u_{1}$ in the first column of $T$ and $v$ bumps $v_{1}$
in the first column of $[u\rightarrow T]$, then $u_{1}\geq u\geq v_{1}\geq v$.
Let $T^{\prime }$ be the tableau obtained by omitting the first column of $T$. Set $%
R_{u_{1}}=R(u_{1}\rightarrow T^{\prime })$ and $R_{v_{1}}=R(v_{1}\rightarrow
[u_{1}\rightarrow T^{\prime }])$, then $R_u=R_{u_{1}}$ and $R_v=R_{v_{1}}$. We have $R_{u_{1}}\geq R_{v_{1}}$ by induction. Thus we obtain
$R_u\geq R_v.$
\end{proof}

\begin{cor}\label{cor 1.2.2} Let $w=a_m\cdots a_1$ be a word and $T$ a tableau of partition shape. Set $R_k=R(a_k\rightarrow[a_{k-1}\cdots a_1\rightarrow T])$. If $w$ is a row, then $R_m\cdots R_1$ is a row.
\end{cor}
\begin{proof} This follows directly from Lemma \ref{lem 1.2.1}.
\end{proof}

\begin{lem}\label{lem 1.2.3}
Let $w=x_{t}x_{t-1}\cdots x_{1}z$ be a word and $T$ a tableau of partition shape. Set
\begin{align*}
&R_{k} =R(x_{k}\rightarrow [x_{k-1}x_{k-2}\cdots x_{1}\rightarrow T])\\
&R_{k}^{\prime } =R(x_{k}\rightarrow [x_{k-1}x_{k-2}\cdots
x_{1}z\rightarrow T]).
\end{align*}
Assume that $z>x$ for any $x\in L(T)\cup \{x_{1},x_{2},\ldots ,x_{t}\}$,
then $R_{k}\geq R_{k}^{\prime }$.
\end{lem}
\begin{proof}
Set $T_{k}=[x_{k}\rightarrow [x_{k-1}x_{k-2}\cdots x_{1}\rightarrow T]]$ and $%
T_{k}^{\prime }=[x_{k}\rightarrow [x_{k-1}x_{k-2}\cdots x_{1}z\rightarrow T]]$%
. Since $z>x$ for any $x\in L(T)\cup \{x_{1},x_{2},\ldots ,x_{t}\}$ by assumptions, it is easy to see that $T_k$ is just the tableau obtained by removing the box with letter $z$ from $T'_k.$

Consider the procedure
\[
x_{k}\rightarrow [x_{k-1}x_{k-2}\cdots x_{1}z\rightarrow T].
\]
If $z$ is
bumped by some letter, then $R_{k}=r$ provided the letter $z$ appears in the $r$%
-th row of $T_{k-1}^{\prime }$, and obviously $R_{k}^{\prime }\leq r=R_{k}$.
If $z$ is not bumped, then $R_{k}=R_{k}^{\prime }$.
\end{proof}

We denote by $\infty $ the symbolic letter which satisfies the relation that $\infty >a$ for any $%
a\in\mathbb{Z}_{>0}$. Then the previous procedure $x\rightarrow T$ makes sense even for $x=\infty.$ We can rewrite Lemma \ref{lem 1.2.3} as follows
\par\smallskip\noindent
\begin{enumerate}
\item[(1.2.1)]
 Let $w=x_{t}x_{t-1}\cdots x_{1}$ be a word and
$T$ a tableau of partition shape. Set
\begin{align*}
&R_{k} =R(x_{k}\rightarrow
[x_{k-1}x_{k-2}\cdots x_{1}\rightarrow T])\\
&R_{k}^{\infty
} =R(x_{k}\rightarrow [x_{k-1}x_{k-2}\cdots x_{1}\infty \rightarrow T]).
\end{align*}
Then $R_{k}\geq R_{k}^{\infty }$.
\end{enumerate}
\begin{prop}\label{prop 1.2.4}
Let $w=y_{t}y_{t-1}\cdots y_{1}x_{t}x_{t-1}\cdots x_{1}$ be a double-row and
$T$  a tableau of partition shape. For
$1\leq k\leq t$, set
\begin{align*}
&R_{k}^{x} =R(x_{k}%
\rightarrow [x_{k-1}\cdots x_{1}\rightarrow T])\\
&R_{k}^{y} =R(y_{k}%
\rightarrow [y_{k-1}\cdots y_{1}x_{t}x_{t-1}\cdots x_{1}\rightarrow T]).
\end{align*} Then $R_{t}^{y}R_{t-1}^{y}\cdots
R_{1}^{y}R_{t}^{x}R_{t-1}^{x}\cdots R_{1}^{x}$ is a double-row.
\end{prop}
\begin{proof}
Let $z(T)$ denote the left most letter of the last row of $T$. We consider the case $x_1\leq z(T)$ and $x_1>z(T)$ separately. \

\par\smallskip\noindent \textbf{Case 1. $x_{1}\leq z(T).$}

We use the induction on $|T|$. For $|T|=1$,
by a direct computation, $R_{k}^{x}=1$ and $R_{k}^{y}=2$ for $1\leq k\leq t$. Hence $%
R_{t}^{y}R_{t-1}^{y}\cdots R_{1}^{y}R_{t}^{x}R_{t-1}^{x}\cdots R_{1}^{x}$ is
a double-row.

Let $T$ be a tableau of partition shape with $|T|=n\geq 1$ such that $x_1\leq z(T).$ By induction hypothesis, we may assume that the assertion holds for any tableau $T'$ of partition shape  with $|T'|<n$ satisfying the condition in Case 1. We write the last row of $T$ as
$$\begin{array}{llll}
z & a_1 & \cdots & a_s
\end{array}$$where $z=z(T).$
Let $T^{\prime }$ be the tableau of partition shape obtained by replacing the
last row of $T$ by%
\[
\begin{array}{lll}
a_{1} & ... & a_{s} \cr
\end{array}.
\]

For a double-row $w=y_{t}y_{t-1}\cdots y_{1}x_{t}x_{t-1}\cdots x_{1}$,
let $\widetilde{w}$ be a double-row defined by $\widetilde{w}=y_{t}y_{t-1}\cdots y_{1}\infty x_{t}x_{t-1}\cdots x_{1}z$%
. For $1\leq k\leq t$, set
\begin{align*}
&\widetilde{R}_{k}^{x} =R(x_{k}\rightarrow [x_{k-1}\cdots
x_{1}z\rightarrow T^{\prime }])\\
&\widetilde{R}_{k}^{y}=R(y_{k}\rightarrow [y_{k-1}\cdots y_{1}\infty x_{t}x_{t-1}\cdots
x_{1}z\rightarrow T])
\end{align*} for $1\leq k\leq t$. Then $\widetilde{R}_{k}^{x}%
=R_{k}^{x}$ since $[z\rightarrow T^{\prime }]=T$.
We see that $\widetilde{R}_{t}^{y}\widetilde{%
R}_{t-1}^{y}\cdots \widetilde{R}_{1}^{y}\widetilde{R}_{t}^{x}\widetilde{%
R}_{t-1}^{x}\cdots \widetilde{R}_{1}^{x}$ is a double-row since $|
T^{\prime }| =n-1$.
By Lemma \ref{lem 1.2.3}, we have $%
R_{k}^{y}\geq \widetilde{R}_{k}^{y}$, hence $$%
R_{t}^{y}R_{t-1}^{y}\cdots R_{1}^{y}R_{t}^{x}R_{t-1}^{x}\cdots
R_{1}^{x}=R_{t}^{y}R_{t-1}^{y}\cdots R_{1}^{y}\widetilde{R}_{t}^{x}%
\widetilde{R}_{t-1}^{x}\cdots \widetilde{R}_{1}^{x}$$ is a double-row.

\par\smallskip\noindent\textbf{Case 2. $x_{1}>z(T).$}

Let $w=y_{t}y_{t-1}\cdots y_{1}x_{t}x_{t-1}\cdots x_{1}$ be a double-row and $T$  a tableau of partition shape such that $x_{1}>z(T).$ Set
\begin{align*}\widetilde{R}_{1}^{x}%
 &=R(x_{1}\rightarrow T)\\
\widetilde{R}_{1}^{y} &=R(y_{1}\rightarrow
(x_{1}\rightarrow T))
\end{align*}
and set
\begin{align*}\widetilde{R}_{k}^{x} &=R(x_{k}\rightarrow
(x_{k-1}\cdots x_{2}y_{1}x_{1}\rightarrow T)),\\
\widetilde{R}_{k}^{y}%
 &=R(y_{k}\rightarrow (y_{k-1}\cdots y_{2}x_{t}\cdots
x_{2}y_{1}x_{1}\rightarrow T)),
\end{align*}
for $2\leq k\leq t.$ It is easy to see  $%
\widetilde{R}_{k}^{x}=R_{k}^{x}$ and $\widetilde{R}_{k}^{y}=R_{k}^{y}$ for $%
1\leq k\leq t$. Since $x_{2}\leq x_{1}<y_{1}$, the double-row $\widetilde{w}%
=y_{t}y_{t-1}\cdots y_{2}x_{t}x_{t-1}\cdots x_{2}$ and the
tableau $T'=[y_{1}x_{1}\rightarrow T]$ satisfy the condition in Case 1.
Hence
\[
R_{t}^{y}R_{t-1}^{y}\cdots R_{2}^{y}R_{t}^{x}R_{t-1}^{x}\cdots
R_{2}^{x}=\widetilde{R}_{t}^{y}\widetilde{R}_{t-1}^{y}\cdots \widetilde{%
R}_{2}^{y}\widetilde{R}_{t}^{x}\widetilde{R}_{t-1}^{x}\cdots \widetilde{%
R}_{2}^{x}
\]
 is a double-row. Noticing that $R_{1}^{y}=R_{1}^{x}+1$, we conclude that $$%
R_{t}^{y}R_{t-1}^{y}\cdots R_{1}^{y}R_{t}^{x}R_{t-1}^{x}%
\cdots R_{1}^{x}$$ is
a double-row.
\end{proof}

\subsection{} Following [B, Chapter 2], we introduce the Knuth's relation $\sim_{K}$ which is the equivalent relation on words, generated by so-called elementary transformations given as follows; for letters $x,y,z$ are letters and words $u,v$,
\begin{eqnarray*}
uxzyv &\sim_K &uzxyv\text{ for \ }x\leq y<z, \\
uyzxv &\sim_K &uyxzv\text{ for \ }x<y\leq z.
\end{eqnarray*}

Let $w$ be a word such that $\mathrm{wt}(w)$ is a partition. Following [Bu, 2.4], we associate a non-negative integer $\mathrm{c}(w),$ called the charge of $w$. Let $T$ be a tableau of partition weight, define its charge $\mathrm{c%
}(T)$ as the charge of the word $w(T)$.

\begin{lem}[{[B, Corollary 2.4.38]}] \label{lem 1.3.1}
If $w$ and $w^{\prime }$ are words whose weights are partitions then

\[
w\sim _{K}w^{\prime }\Longrightarrow \mathrm{c}(w)=\mathrm{c}(w^{\prime })
\]
\end{lem}

\subsection{}Here we give a brief explanation of Sch\"{u}tzenberger's jeu de taquin sliding algorithm. See [F, Section 2], [Bu, Chapter 2] for more details.
Let $\lambda -\rho$ be a skew shape. A position $b$ not belonging to $\lambda -\rho$ is called an upper (resp. lower) jeu de taquin position if $b$ shares at least a lower (resp. upper) edge with $\lambda -\rho$ and $\{b\}\cup (\lambda -\rho)$ is a valid skew shape.

Suppose we are given a tableau $T$ of the skew shape $\lambda -\rho$. To each upper (resp. lower) jeu de taquin position $b$, we associate a transformation $\mathrm{jdt}_b(T)$ of $T$, called a jeu de taquin slide of $T$ into $b$, by the following rule.
\begin{enumerate}
\item[(1)]
 Place a pawn $\clubsuit$ (resp. $\spadesuit$) in the position $b$.

\item[(2)]
 Compare the letters directly to the right (resp. left) and directly below (resp. above) the pawn. Exchange the pawn with the letter so that the result does not violate strict increasing along columns or weak increasing along rows.

If there is a letter below (resp. above) the pawn but no letter to the right (resp. left) of the pawn, exchange the pawn with the letter below (resp. above). Similarly, if there is a letter to the right (resp. left) of the pawn, but no letter below (resp. above) the pawn, then exchange the pawn with the letter to the right (resp. left).

\item[(3)] Repeat (2), until there is no letter to the right (resp. left) of or below (resp. above) the pawn. $\mathrm{jdt}_b(T)$ is the  tableau obtained by omitting the pawn from the result.
\end{enumerate}

\begin{exa} \emph{Let} $T$ \emph{be the  tableau}
\[
\begin{array}{llll}
& a & 2 & 3 \\
& 1 & 4 & 6\\
1 & 5 & b &\\
3 & & &
\end{array}
\]\emph{ with be the upper (resp. lower)jeu de taquin position} $a$ \emph{(}\emph{resp}. $b$\emph{)}.
\emph{We have}

$$\begin{array}{llll}
& \clubsuit & 2 & 3 \\
& 1 & 4 & 6\\
1 & 5 & &\\
3 & & &
\end{array} \rightarrow
\begin{array}{llll}
& 1 & 2 & 3 \\
& \clubsuit & 4 & 6\\
1 & 5 &  &\\
3 & & &
\end{array} \rightarrow
\begin{array}{llll}
& 1 & 2 & 3 \\
& 4 & \clubsuit & 6\\
1 & 5 &  &\\
3 & & &
\end{array} \rightarrow
\begin{array}{llll}
& 1 & 2 & 3 \\
& 4 & 6 & \clubsuit\\
1 & 5 &  &\\
3 & & &
\end{array}$$ \emph{and}
$$\begin{array}{llll}
&  & 2 & 3 \\
& 1 & 4 & 6\\
1 & 5 & \spadesuit  &\\
3 & & &
\end{array} \rightarrow
\begin{array}{llll}
&  & 2 & 3 \\
& 1 & 4 & 6\\
1 & \spadesuit & 5  &\\
3 & & &
\end{array} \rightarrow
\begin{array}{llll}
&  & 2 & 3 \\
& \spadesuit & 4 & 6\\
1 & 1 & 5  &\\
3 & & &
\end{array}
$$ \emph{Hence, we obtain}
$\mathrm{jdt}_a(T)=\begin{array}{llll}
& 1 & 2 & 3 \\
& 4 & 6 & \\
1 & 5 &  &\\
3 & & &
\end{array}$ \emph{and} $\mathrm{jdt}_b(T)=\begin{array}{llll}
&  & 2 & 3 \\
&  & 4 & 6\\
1 & 1 & 5  &\\
3 & & &
\end{array}$
\end{exa}

Two tableaux $T$ and $T'$ are called jeu de taquin equivalent if $T'$ can be obtained by a sequence of jeu de taquin slides of $T.$
Let $T$ be a tableau. It is well-known  that, see [F, 2], for example, there exists a unique tableau $\mathrm{jdt}(T)$ of partition shape which is jeu de taquin equivalent to $T.$     The following two results are well-known
\begin{thm}[{[B, Remark 2.3.25]}]\label{thm 1.4.2}
 Let $T$ be a tableau. Then
\begin{equation}\label{F 1.1}
\mathrm{jdt}(T)=[w(T)\rightarrow \emptyset]
\end{equation}
\end{thm}

\begin{thm}[{[F, Lemma 2.3]}]\label{thm 1.4.3} Let $T$ be a tableau, and let $a$ be a jeu de taquin position of $T$. Then
$$w(T)\sim_K w(\mathrm{jdt}_a(T)).$$
\end{thm}
\begin{rem}
\emph{In [F, Lemma 2.3], Fomin proved Theorem 1.4.3 in the case where  }$T$\emph{ is a standard tableau. But his proof works when }$T$\emph{ is any tableau.}
\end{rem}
As a corollary of Theorem 1.4.3, we have
\begin{cor}\label{cor 1.4.4}
Let $T$ be a tableau of partition weight. Then
$$\mathrm{c}(T)=\mathrm{c}(\mathrm{jdt}(T)).$$
\end{cor}

\begin{proof}
By Theorem \ref{thm 1.4.3}, we have $w(T)\sim_K w(\mathrm{jdt}(T))$. By Lemma \ref{lem 1.3.1}, we have $\mathrm{c}(T)=\mathrm{c}(w(T))=\mathrm{c}(w(\mathrm{jdt}(T)))=\mathrm{c}(\mathrm{jdt}(T)).$
\end{proof}

\subsection{}
Following [F], we introduce the RS-correspondence, which is a way to associate with a word $w$ a pair $(Q_{w},P_{w})$ of  tableaux of the same partition shape. Assume $w=a_ma_{m-1}\cdots a_1$ is given. Then $(Q_{w},P_{w})$ is obtained as follows.
Begin with $(Q(0),P(0))=(\emptyset,\emptyset),$ and assume that $(Q(t),P(t))$ are defined for $t<m.$ Define $Q(t+1)$ and $P(t+1)$ by
\par\smallskip\noindent
\begin{enumerate}
\item[(1)]
$P(t+1)=[a_{t+1}\rightarrow P(t)],$
\item[(2)]
$Q(t+1)$ is obtained from $Q(t)$ by adding a box with letter $t+1$ in it such that $Q(t+1)$ and $P(t+1)$ have the same shape.
\end{enumerate}
\par\smallskip\noindent
The process ends at $(Q(m),P(m))$. Let $Q_w=Q(m)$ and $P_w=P(m).$ We denote the correspondence by $w\overset{\mathrm{RS}}{\rightarrow }(Q_{w},P_{w}).
$
Note that $Q_w$ is a standard tableau, called the recording tableau. Notice that the RS-correspondence asserts that for a given pair of tableaux $(Q,P)$ of the same partition shape such that $Q$ is a standard, there exists a unique word $w$ such that $w\overset{\mathrm{RS}}{\rightarrow }(Q,P).$

\subsection{} A word $w=a_{m}a_{m-1}\cdots a_{1}$ is called a lattice permutation if for $1\leq r\leq m$ and $i\geq 1$, the number of occurrences of the symbol $i$ in $a_{r}a_{r-1}\cdots a_{1}$ is not less than the number of occurrences of the symbol $i+1$.
For partitions $\lambda$, $\rho$ and $\mu$ with $\lambda\supset \rho$, let $\mathrm{Tab}(\lambda-\rho,\mu)$ be the set of tableaux of shape $\lambda-\rho$ and weight $\mu$. Denote by $\mathrm{Tab}^0(\lambda-\rho,\mu)$ the set of tableau $T\in \mathrm{Tab}(\lambda-\rho,\mu)$ such that $w(T)$ is a lattice permutation.

Let $\mu\in\mathscr{P}_{N},$ and let $\lambda$, $\rho$ be partitions such that $\lambda\supset\rho$ and $|\lambda|-|\rho|=N.$ In [M, I, (9.4)], a bijection
\begin{equation}\label{nF 1.2}
\theta: \mathrm{Tab}(\lambda -\rho ,\mu ) \isom \coprod\limits_{\nu
\in \mathscr{P}_{N}}\mathrm{Tab}^{0}(\lambda
- \rho ,\nu )\times \mathrm{Tab}(\nu ,\mu )
\end{equation}
is constructed.  In what follows we shall construct a different type of bijection apart
from $\theta$, by using RS-correspondence.
For $\nu\in\mathscr{P}_{N}$, we denote by $L(\nu)$  the set of lattice permutations of weight $\nu$. To $T\in\mathrm{ST}(\nu)$, we associate a word $\sigma_T=b_{N}\cdots \cdots b_1$ where $b_k=i$ if the letter $k$ appears in the $i$-th row of $T.$ It is well-known that the map
$T\mapsto \sigma_T$ gives a bijection $L(\nu)\isom \mathrm{ST}(\nu).$

For $T\in\mathrm{Tab}(\lambda -\rho ,\mu )$, set $Q_T=Q_{w(T)}$, $P_T=P_{w(T)}$ and $\sigma(T)=\sigma_{Q_T}.$ We write $w(T)=a_{N}\cdots a_1$ and set $b_{i}=R(a_{i}\rightarrow [a_{i-1}\cdots a_{1}\rightarrow \emptyset ])$. Then it is easy to see $\sigma(T)=b_{N}\cdots b_1.$

\begin{lem}\label{lem 1.6.1}
For $T\in \mathrm{Tab}(\lambda - \rho ,\mu ),$ $\sigma (T)$
is compatible with $\lambda -\rho .$
\end{lem}
\begin{proof}
Under the notation above, it is enough to show that if $i$-th row and $(i+1)$-th row of $T$ are given as
$$\begin{array}{llllllll}
&  & a_{p} & a_{p-1} & \cdots  & a_s  & \cdots  & a_{k} \\
a_{q} & \cdots  & a_{t} & a_{t-1} & \cdots  & a_{p+1} &  &
\end{array},$$ then the array
$$T'_i=\begin{array}{llllllll}
&  & b_{p} & b_{p-1} & \cdots  & b_s  & \cdots  & b_{k} \\
b_{q} & \cdots  & b_{t} & b_{t-1} & \cdots  & b_{p+1} &  &
\end{array}$$ is a tableau.

Set $w_1=a_{p}a_{p-1}\cdots a_k,$ $w_2=a_qa_{q-1}\cdots a_{p+1}$ and $w_3=a_t a_{t-1}\cdots a_{p+1}a_{p}\cdots a_s.$ Note that $w_1$ and $w_2$ are both rows and $w_3$ is a double-row.
Set $T_1=[a_{k-1}\cdots a_1\rightarrow \emptyset],$ $T_2=[a_p\cdots a_1\rightarrow \emptyset]$ and $T_3=[a_{s-1}\cdots a_1\rightarrow \emptyset].$ Consider $w_1\rightarrow T_1$ and $w_2\rightarrow T_2.$ Then we see that $b_{p}\cdots b_k$ and $b_{q}\cdots b_{p+1}$ are rows by Corollary \ref{cor 1.2.2}. Consider $w_3\rightarrow T_3.$ Then $b_t\cdots b_{p+1}b_p\cdots b_s$ is a double-row by Proposition \ref{prop 1.2.4}. Hence $T'_i$ is a tableau.
\end{proof}

Take $T\in\mathrm{Tab}(\lambda -\rho ,\mu )$ and assume that the shape of $Q_T$ is $\nu$. Then by Lemma 1.6.1, there exists a unique tableau $D_T$ such that $w(D_T)=\sigma_T.$

\begin{thm}\label{thm 1.6.3}
Let $\mu\in\mathscr{P}_{N},$ and let $\lambda$, $\rho$ be partitions such that $\lambda\supset\rho$ with $|\lambda|-|\rho|=N.$ Then
\begin{enumerate}
\item[(1)]
The map $T \mapsto (D_T,P_T)=(D_T,\mathrm{jdt}(T))$ gives a bijection
\begin{eqnarray*}\label{F 1.2}
\widetilde{\Gamma}: \mathrm{Tab}(\lambda -\rho ,\mu ) \rightarrow \coprod\limits_{\nu \in
\mathscr{P}_{N }}\mathrm{Tab}^{0}(\lambda - \rho ,\nu )\times
\mathrm{Tab}(\nu ,\mu )\\
\end{eqnarray*}
\item[(2)]
If we write $\widetilde{\Gamma}(T)=(D,S),$ then $\mathrm{c}(T)=\mathrm{c}(S).$
\end{enumerate}
\end{thm}
\begin{proof}
For $T\in\mathrm{Tab}(\lambda -\rho ,\mu )$, $\widetilde{\Gamma}(T)=(D_{T},P_T)$ is obtained by the following procedure $$ T\mapsto w(T) \mapsto (Q_T,P_T)\mapsto (\sigma(T),P_T)\mapsto (D_T,P_T).$$ Each step is one to one, which implies that $\widetilde{\Gamma}$ is injective. Since there exists a bijection $\theta$ in (\ref{nF 1.2}), both sides have the same cardinality. Hence $\widetilde{\Gamma}$ is a bijection. Note that $P_T=\mathrm{jdt}(T)$ by Theorem \ref{thm 1.4.2}. This proves (1).
Since $P_T=\mathrm{jdt}(T),$  we have $\mathrm{c}(T)=\mathrm{c}(P_T)$ by Corollary \ref{cor 1.4.4}. Thus (2) holds.
\end{proof}
\begin{rem}\label{rem 1.6.4}
\emph{The property (2) is used in the later discussions. It is likely, as many
examples show, that} $\theta=\widetilde{\Gamma}$\emph{, but} \emph{we don't know the proof. The author does
not know whether the property (2) holds for the map} $\theta$\emph{ , though it is stated
in [KR1] without proof that} $\mathrm{c}(T)=\mathrm{c}(S)$ \emph{for} $\theta(T)=(D,S).$
\end{rem}

\section{Double Kostka polynomials}
In this section, we introduce the double Kostka polynomials, following [LS]. For more information, see [Sh1], [Sh2] and [LS].

\subsection{}
A pair of partitions $\Bla=(\lambda',\lambda'')$ is called a double partition of $N$ if $|\lambda'|+|\lambda''|=N.$ We denote by $\mathscr{P}_{N,2 }$ the set of double partitions of $N$. For partitions $\lambda,\mu,$ the Koskta polynomial $K_{\lambda,\mu}(t)\in \mathbb{Z}[t]$ is defined as in
[M, III,6]. In [S1], [S2], Kostka functions associated to complex reflection
groups were introduced as a generalization of Koskta polynomials. As a
special case of such Kostka functions, double Kostka polynomials $K_{\Bla,\Bmu}$ are defined, which are polynomials in $\mathbb{Z}[t]$ indexed by double partitions $\Bla,\Bmu\in\mathscr{P}_{N,2 }.$ A relationship between double Kostka polynomials and usual Koskta polynomials was studied in [LS]. In particular, in the special case where $\Bmu=(-,\mu''),$ there exists a very simple formula as follows.
\begin{prop}[{[LS, Lemma 3.4]}]\label{nthm 2.1.1}
 Assume that $\Bla,\Bmu\in\mathscr{P}_{N,2 }$ such that $\Bmu=(-,\mu'').$ Then we have
\begin{equation}\label{nF 2.1}
K_{\Bla,\Bmu}(t)=t^{|\lambda'|}\sum_{\eta\in\mathscr{P}_{N}}c^{\eta}_{\lambda',\lambda''}K_{\eta,\mu''}(t^2),
\end{equation} where $\Bla=(\lambda',\lambda'')$ and $c^{\eta}_{\lambda',\lambda''}$ is the Littlewood-Richardson coefficient (cf.
[M, I, 9]).
\end{prop}

\subsection{}Let $\Bla=(\lambda',\lambda'')\in \mathscr{P}_{N,2}.$ Following [LS], we call a pair $T=(T_{+},T_{-})$ a tableau of shape $\Bla$ if $T_{+}$ (resp. $T_{-}$) is a tableau of shape $\lambda'$ (resp. $\lambda''$). Let $\mathrm{Tab}(\Bla)$ be the set of tableaux of shape $\Bla.$

Let $\Bla=(\lambda',\lambda'')\in \mathscr{P}_{N,2}$ with $\lambda'=(\lambda'_1,\ldots,\lambda'_s)$ and $\lambda''=(\lambda''_1,\ldots,\lambda''_t)$. For an integer $a\geq \lambda''_1,$ we define a partition $\widetilde{\xi}_{\Bla,a}=(\xi_1,\xi_2,\ldots,\xi_{s+t})\in \mathscr{P}_{N+as}$ by

\begin{equation*}
\xi_i = \begin{cases}
                \la'_i + a  &\quad\text{ for $1 \le i \le s$}, \\
                \la''_{i - s}        &\quad\text{ for $s+1 \le i \le s+t$. }
        \end{cases}
\end{equation*}
and define a skew shape $\xi_{\Bla,a}$ by $\xi_{\Bla,a}=\widetilde{\xi}_{\Bla,a}-(a^s).$ Then $\xi_{\Bla,a}$ consists of two connected components of shape $\lambda'$ and $\lambda''$. Thus $T\in\mathrm{Tab}(\Bla)$ can be naturally identified with an element $\widetilde{T}_a\in\mathrm{Tab}(\xi_{\Bla,a}).$
For $T\in \mathrm{Tab}(\Bla)$, put
$w(T)=w(T_{-})*w(T_{+})$ and $\mathrm{wt}(T)=\mathrm{wt}(w(T)).$ It is easy to see that
$w(T)=w(\widetilde{T}_a)$
and $\mathrm{wt}(T)=\mathrm{wt}(\widetilde{T}_a).$

Let $\mu\in\mathscr{P}_{N }.$ We denote by $\mathrm{Tab}(\Bla,\mu)$ the set of tableaux of shape $\Bla$ and weight $\mu.$ Let $\mathrm{Tab}^0(\Bla,\mu)$ the subset of those $T$ such that $w(T)$ is a lattice permutation.
By applying Theorem \ref{thm 1.6.3} for $\lambda-\rho=\xi_{\Bla,a},$ we obtain a bijection $\widetilde{\Gamma}.$ Under the identification $\mathrm{Tab}(\xi_{\Bla,a},\mu)$ with $\mathrm{Tab}(\Bla,\mu),$ $\widetilde{\Gamma}$ induces a bijection
\begin{equation}\label{nF 2.2}
\Gamma_a: \mathrm{Tab}(\Bla,\mu ) \rightarrow \coprod\limits_{\nu \in
\mathscr{P}_{N }}\mathrm{Tab}^{0}(\Bla ,\nu )\times
\mathrm{Tab}(\nu ,\mu ).
\end{equation}
But since $w(T)=w(\widetilde{T}_a)$ for any $a,$ $\Gamma_a$ does not depend on the choice of $a,$ which we denote by $\Gamma.$

\subsection{} It is well-known that the Kostka polynomials $K_{\lambda,\mu}$ have a combinatorial
description due to Lascoux and Sch\"{u}tzenberger (see e.g., [M, III, (6.5)]).
\begin{thm}[Lascoux-Sch\"{u}tzenberger] \label{nthm 2.3.1}
For partitions $\lambda,\mu\in\mathscr{P}_{N },$ we have
 \begin{equation}\label{nF 2.3}
 K_{\lambda,\mu}(t)=\sum_{T\in\mathrm{Tab}(\lambda,\mu)}t^{\mathrm{c}(T)}.
 \end{equation}
\end{thm}

For $\Bla\in\mathscr{P}_{N,2 },$ we define the charge $\mathrm{c}(T)$ of $T$ by $\mathrm{c}(T)=\mathrm{c}(w(T)).$ Then we
have the following Lascoux-Sch\"{u}tzenberger type formula for double Kostka polynomials $K_{\Bla,\Bmu}(t)$ for the special case where $\Bmu=(-,\mu'').$
\begin{thm}\label{nthm 2.3.2}
Let $\Bla,\Bmu\in\mathscr{P}_{N,2 },$ and assume that $\Bmu=(-,\mu'').$ Then
\[
K_{\Bla,\Bmu}(t)=t^{|\lambda'|}\sum_{T\in\mathrm{Tab}(\Bla,\mu'')}t^{2\mathrm{c}(T)}.
\]
\end{thm}
\begin{proof}
In [LS, Theorem 3.12], a similar formula was proved by making use of the bijection $\theta$ in (\ref{nF 1.2}) instead of $\widetilde{\Gamma}.$ Note that in that case, the charge $\mathrm{c}(T)$ defined there may be different from the charge here. Although the theorem can be proved in a similar way as in [LS], we give the proof below for the sake of completeness.
Assume that $\Bla=(\lambda',\lambda'')\in \mathscr{P}_{N,2 }$ and $\nu\in\mathscr{P}_{N}.$ First we note, by a general argument (see [LS, Corollary 3.9], [M, I,9]), that the bijection (\ref{nF 1.2}) implies that
\begin{equation}\label{nF 2.4}
|\mathrm{Tab}^0(\Bla,\nu)|=c^{\nu}_{\lambda',\lambda''}.
\end{equation}
We consider the map $p:\mathrm{Tab}(\Bla,\mu'')\rightarrow \amalg_{\nu\in\mathscr{P}_{N}}\mathrm{Tab}(\nu,\mu'')$ defined by $p(T)=S$ for $\Gamma(T)=(D,S).$ Then by (\ref{nF 2.2}), for each $S\in\mathrm{Tab}(\nu,\mu''),$ the set $p^{-1}(S)$ has the cardinality $c^{\nu}_{\lambda',\lambda''}$ and by Theorem \ref{thm 1.6.3} (2), $\mathrm{c}(T)=\mathrm{c}(S)$ for each $T\in p^{-1}(S).$ Hence we have
\begin{eqnarray*}
\sum_{T\in\mathrm{Tab}(\Bla,\mu'')}t^{\mathrm{c}(T)} &=& \sum_{\nu\in\mathscr{P}_{N}}\sum_{S\in\mathrm{Tab}(\nu,\mu'')}c^{\nu}_{\lambda',\lambda''}t^{\mathrm{c}(T)} \\
&=& \sum_{\nu\in\mathscr{P}_{N}}c^{\nu}_{\lambda',\lambda''}K_{\nu,\mu''}(t).
\end{eqnarray*} The second formula follows from (\ref{nF 2.3}). Now the theorem follows from Theorem
\ref{nthm 2.1.1}.
\end{proof}

\section{Crystals $B^{1,m}$}
\subsection{} Here we briefly review the notion of crystals. See, e.g., [HK] for more details.
Let $\mathfrak{g}$ be a symmetrizable Kac-Moody algebra with index set $I,$ weight lattice $P,$ dual weight lattice $P^\vee,$ root lattice $Q,$ fundamental weights $\Lambda_i$($i\in I$), simple roots $\alpha_i$ ($i\in I$) and simple coroots $h_i$($i\in I$). Let $\langle\text{ },\text{ }\rangle:P^\vee\times P\rightarrow \mathbb{Z}$ be the pairing defined by $\langle h, \Lambda\rangle=\Lambda(h)$ for $h\in P^\vee,$ $\Lambda\in P.$
Let $U_{q}(\mathfrak{g})$ be the quantized universal enveloping algebra of $\mathfrak{g}.$ An (abstract) $U_{q}(\mathfrak{g})$-crystal is a non-empty set $B$ together with maps
$$\mathrm{wt}:B\rightarrow P,\text{ }\varepsilon _{i},\varphi_{i}:B\rightarrow\mathbb{Z}\cup\{-\infty\},\text{ }e_{i},f_{i}:B\rightarrow B\sqcup \{0\}$$
subject to the conditions

\smallskip
\begin{enumerate}
\item[(1)]

$\left\langle h_{i},\mathrm{wt}(b)\right\rangle =\varphi_{i}(b)-\varepsilon _{i}(b)$ for all $i\in I$,

\item[(2)]
$\mathrm{wt}(e_{i}(b))=\mathrm{wt}(b)+\alpha _{i}$ if $e_{i}(b)\in B
$

\item[(3)]
 $\mathrm{wt}(f_{i}(b))=\mathrm{wt}(b)-\alpha _{i}$ if $f_{i}(b)\in B
$

\item[(4)]
 $\varepsilon _{i}(e_{i}(b))=\varepsilon _{i}(b)-1$, $\varphi
_{i}(e_{i}(b))=\varphi _{i}(b)+1$ if $e_{i}(b)\in B$

\item[(5)]
 $\varepsilon _{i}(f_{i}(b))=\varepsilon _{i}(b)+1$, $\varphi
_{i}(f_{i}(b))=\varphi _{i}(b)-1$ if $f_{i}(b)\in B$

\item[(6)]
 for $b,b^{\prime }\in B$, $f_{i}(b)=b^{\prime }$ if and only if $%
e_{i}(b^{\prime })=b$

\item[(7)]
 if $\varphi _{i}(b)=-\infty $ for $b\in B$, then $e_{i}(b)=f_{i}(b)=0$.
\end{enumerate}
\smallskip
\par\noindent
The Operators $e_i$ and $f_i$ are referred to as Kashiwara operators.
Let $B$ be an $U_{q}(\mathfrak{g})$-crystal. If $b'=f_{i}(b)$ for $b,b' \in B$, we draw an arrow as $b\overset{i}{\rightarrow }b^{\prime }.$ In this way, $B$ gets endowed with the structure of $I$-colored oriented graph, called the crystal graph. $B$ is connected if its crystal graph is connected. A crystal $B$ is said to be semi-regular if, for all $b\in B$, $i\in I$
\begin{align}\varepsilon_{i}(b) &=\max\{k\in \mathbb{Z}_{\geq 0}\mid e_{i}^k(b)\neq 0\}\\
\varphi_{i}(b) &=\max\{k\in \mathbb{Z}_{\geq 0}\mid f_{i}^k(b)\neq 0\}.
\end{align}

\subsection{}Let $B_1$ and $B_2$ be two crystals. A crystal morphism $\psi:B_1\rightarrow B_2$ is a map $B_1\sqcup\{0\}\rightarrow B_2\sqcup\{0\}$ such that

\begin{enumerate}
\item[(1)]
$\psi(0)=0;$

\item[(2)]
 if $b\in B_1$ and $\psi(b)\in B_2,$ then $\mathrm{wt}(\psi(b))=\mathrm{wt}(b),$ $\varepsilon_i(\psi(b))=\varepsilon_i(b),$ and $\varphi_i(\psi(b))=\varphi_i(b);$

\item[(3)]
 for $b\in B_1,$ $\psi(e_ib)=e_i\psi(b)$ if $\psi(e_ib)\neq 0$ and $e_i\psi(b)\neq 0;$

\item[(4)]
 for $b\in B_1,$ $\psi(f_ib)=f_i\psi(b)$ if $\psi(f_ib)\neq 0$ and $f_i\psi(b)\neq 0.$
\end{enumerate}

\par\noindent
A morphism $\psi:B_1\rightarrow B_2$ is called an embedding if the induced map $\psi:B_1\sqcup\{0\}\rightarrow B_2\sqcup\{0\}$ is an injection and commutes with the Kashiwara operators. Moreover, $\psi$ is called an isomorphism if the induced map is a bijection.

\subsection{}For crystals $B_{1}$ and $B_{2}$, we
introduce the tensor product $B_{2}\otimes B_{1}$, which owns a crystal structure. As a set, it coincides
with the cartesian product $B_{2}\times B_{1}$ and the crystal structure is given
by

\begin{align*}e_{i}(b_{2}\otimes
b_{1}) &=\left\{
\begin{array}{ll}
e_{i}(b_{2})\otimes b_{1} &\mbox{if $\varepsilon _{i}(b_{2})>\varphi _{i}(b_{1})$}\\
b_{2}\otimes e_{i}(b_{1}) &\mbox{if $\varepsilon _{i}(b_{2})\leq \varphi
_{i}(b_{1})$,}
\end{array}
\right.\\
f_{i}(b_{2}\otimes
b_{1}) &=\left\{
\begin{array}{ll}
f_{i}(b_{2})\otimes b_{1} &\mbox{if $\varepsilon _{i}(b_{2})\geq \varphi
_{i}(b_{1})$,}\\
b_{2}\otimes f_{i}(b_{1}) &\mbox{if $\varepsilon _{i}(b_{2})<\varphi _{i}(b_{1})
$.}
\end{array}
\right.\\
\varepsilon_{i}(b_{2}\otimes b_{1}) &=\max(\varepsilon_{i}(b_1),\varepsilon_{i}(b_1)+\varepsilon_{i}(b_2)-\varphi_i(b_1)),\\
\varphi_{i}(b_{2}\otimes b_{1}) &=\max(\varphi_{i}(b_1),\varphi_{i}(b_1)+\varphi_{i}(b_2)-\varepsilon_i(b_1)),\\
\mathrm{wt}(b_2\otimes b_1) &=\mathrm{wt}(b_2)+\mathrm{wt}(b_1).
\end{align*}
Here $0\otimes b$ and $b\otimes 0$ are understood to be $0$. We remark that we use the opposite of the Kashiwara's tensor product convention here, see, e.g., [HK, 4.5].
\par

For semi-regular crystals $B_{i}$ ($i=1,\ldots ,r$), we consider the tensor product crystal $B=B_{r}\otimes B_{r-1}\otimes
\cdots \otimes B_{1} .$ Fix $i\in I.$ To each $b=b_{r}\otimes b_{r-1}\otimes \cdots \otimes
b_{1}\in B,$ we assign a  sequence of the symbols $+$ and $-$:
$$\underset{\varphi _{i}(b_{r})}{\underbrace{+\cdots \cdots +}}\underset{%
\varepsilon _{i}(b_{r})}{\underbrace{-\cdots \cdots -}}\cdots \cdots \cdots
\underset{\varphi _{i}(b_{1})}{\underbrace{+\cdots \cdots +}}\underset{%
\varepsilon _{i}(b_{1})}{\underbrace{-\cdots \cdots -}},$$ called the $i$-signature of $b.$ The
reduced $i$-signature of $b$ is obtained by removing subsequence $-+$ of
the $i$-signature of $b$ repeatedly until it becomes the form: $+\cdots
\cdots +-\cdots \cdots -.$  Then $\varphi _{i}(b)$ equals the number of $+$
in the reduced $i$-signature and $\varepsilon _{i}(b)$ equals the number of $%
-$ in the reduced $i$-signature. If $\varepsilon _{i}(b)=0$, then $e_{i}(b)=0
$, otherwise $e_{i}$ acts on $b_{j_{1}}$ which gives the leftmost $-$ in the
reduced $i$-signature: $$e_{i}(b)=b_{r}\otimes \cdots \otimes e_{i}(b_{j_{1}})\otimes
\cdots \otimes b_{1}.$$
If $\varphi _{i}(b)=0$, then $f_{i}(b)=0$,
otherwise $f_{i}$ acts on $b_{j_{2}}$ which gives the rightmost $+$ in the
reduced $i$-signature: $$f_{i}(b)=b_{r}\otimes \cdots \otimes
f_{a}(b_{j_{2}})\otimes \cdots \otimes b_{1}.$$

\subsection{} In the rest of this paper, we assume that $\mathfrak{g}$ is the affine Kac-Moody algebra of type $A^{(1)}_{n-1}$ with index set $I=\{0,1,\ldots,n-1\}$, and $\mathfrak{g}_{0}$ is the corresponding finite dimensional simple Lie algebra obtained by removing the $0$ node of the Dynkin diagram of $\mathfrak{g}$. Let $U'_{q}(\mathfrak{g})$ be the subalgebra of $U_{q}(\mathfrak{g})$ defined by dropping $q^{d}$ ($d$ is the degree operator) from the generators. Then there exist a natural inclusions $$U_{q}(\mathfrak{g})\supset U'_{q}(\mathfrak{g})\supset U_{q}(\mathfrak{g}_{0}).$$
We can choose
\begin{equation}
P_{\mathrm{cl}}=\mathbb{Z}\Lambda_{0}\oplus\mathbb{Z}\Lambda_{1}\oplus\cdots\oplus\mathbb{Z}\Lambda_{n-1}
\end{equation}
 as the weight lattice with respect to $U'_{q}(\mathfrak{g})$. One can define a notion of $U'_{q}(\mathfrak{g})$-crystal by replacing $P$ in the definition of $U_{q}(\mathfrak{g})$-crystal by $P_{\mathrm{cl}}.$
Set $I_{0}=I\backslash \{0\}.$ Denote the fundamental weights of $\mathfrak{g}_{0}$ by $\overline{\Lambda}_{i}$ ($i\in I_{0}$) and the weight lattice of $U_{q}(\mathfrak{g}_{0})$ by $\overline{P}$, then
\begin{equation}
\overline{P}=\mathbb{Z}\overline{\Lambda}_{1}\oplus\mathbb{Z}\overline{\Lambda}_{2}\oplus\cdots\oplus\mathbb{Z}\overline{\Lambda}_{n-1}.
\end{equation}
 Denote by $\overline{P}^+$ the subset of dominant weights.

Define a map $\mathbb{Z}^n\rightarrow\overline{P}$ by $(a_1,\ldots,a_n)\mapsto (a_1-a_2)\overline{\Lambda}_1+\cdots+(a_{n-1}-a_{n})\overline{\Lambda}_{n-1},$ then it induces an isomorphism $\mathbb{Z}^n/\mathbb{Z}(1,\ldots,1)\isom \overline{P}.$ Let $\epsilon_i$ ($1\leq i\leq n$) be the $i$-th unit vector in  $\mathbb{Z}^n,$ and $\overline{\epsilon}_i$ the image of $\epsilon_i.$  Then we have
\begin{align}
&\overline{\Lambda}_i=\overline{\epsilon}_1+\overline{\epsilon}_2+\cdots+\overline{\epsilon}_i, \text{ for } i\in I_0,\label{3.4.3}\\
&\overline{\epsilon}_1+\overline{\epsilon}_2+\cdots+\overline{\epsilon}_n=0.
\end{align}
For $\lambda\in \mathscr{P}_N$ such that  $l(\lambda) \leq n$, we write it as $\lambda=(\lambda_{1},\ldots,\lambda_{n})\in \mathbb{Z}^n$ with $\lambda_{i}\geq 0.$ Then $\lambda$ determines an element $\sum\limits_{i=1}^{n}\lambda_i\overline{\epsilon}_i=\sum\limits_{i=1}^{n-1}(\lambda_i-\lambda_{i+1})\overline{\Lambda}_i\in\overline{P}^+.$  We also denoted it by $\lambda$ by abuse of the notation.   Let $\boldsymbol{\lambda}=(\lambda',\lambda'')$ be a double partition with $\lambda'=(\lambda'_{1},\ldots,\lambda'_{s}),\lambda''=(\lambda''_{1},\ldots,\lambda''_{t})$ such that $s+t\leq n$. Then $\boldsymbol{\lambda}$ determines an element
\begin{equation*}
\begin{split}
&\sum\limits_{i=1}^{s}\lambda'_i\overline{\epsilon}_i+\sum\limits_{j=1}^{t}\lambda''_j\overline{\epsilon}_{s+j}\\
=&\sum_{i=1}^{s-1}(\lambda'_i-\lambda'_{i+1})\overline{\Lambda}_i+(\lambda'_s-\lambda''_1)\overline{\Lambda}_s+\sum_{j=1}^{t-1}(\lambda''_j-\lambda''_{j+1})\overline{\Lambda}_{s+j}\in \overline{P},
\end{split}
\end{equation*}
which we also denote it by $\Bla.$ Note that $\boldsymbol{\lambda}\in \overline{P}$ is not necessarily a dominant weight.

\subsection{} The Kirillov-Reshetikhin modules (KR modules for short) $W^{(r)}_m$ ($r\in I_0, m\geq 1$) are $U'_{q}(\mathfrak{g})$-modules  first introduced by Kirillov and Reshetikin in \cite{KR2}. As a $U_{q}(\mathfrak{g}_0)$-module, $W^{(r)}_m$ is isomorphic to the irreducible highest weight module $V(\lambda)$ with respect to the highest weight $\lambda=m\overline{\Lambda}_r.$
It is known that $W^{(r)}_m$ has a crystal basis, denoted by $B^{r,m}.$ As a $U_{q}(\mathfrak{g}_0)$-crystal, $B^{r,s}$ is isomorphic to the crystal basis $B(\lambda)$ of the irreducible highest weight module $V(\lambda),$ $\lambda=m\overline{\Lambda}_r.$
\par
Let us focus on $B^{1,m}$ and its $U_{q}(\mathfrak{g}_0)$-crystal. We give a brief discussion on the tableau representation of $B^{1,m}.$ For more details, refer to [O].

\par

First we introduce $B^{1,1}$ . The crystal graph of $B^{1,1}$ is given as
$$\begin{array}{l}
1%
\end{array}%
\overset{1}{\rightarrow }%
\begin{array}{l}
2%
\end{array}%
\overset{2}{\rightarrow }\cdots \overset{n-1}{\rightarrow }%
\begin{array}{l}
n%
\end{array}%
$$ and $\mathrm{wt}(i
)=\epsilon _{i},$ where $\epsilon_i$ is the weight defined in 3.2. Next we consider the general case $B^{1,m}.$
As a set, $B^{1,m}$ is the set of all tableaux of shape $(m)$ with letters $1,2,\ldots, n$, namely, $$B^{1,m}=\{x_1x_2\cdots x_m
\mid 1\leq x_{1}\leq x_{2}\leq \cdots \leq x_{m}\leq n\}.$$
We have a natural embedding of crystals
\begin{equation}
     B^{1,m}\hookrightarrow (B^{1,1})^{\otimes m},\text{ }
x_1x_2\cdots x_m \mapsto x_1\otimes x_2\otimes\cdots \otimes x_m.
\end{equation}

\subsection{} For $\mu=(\mu_1,\ldots,\mu_r)\in \mathscr{P}_N,$ we set $B(\mu)=B^{1,\mu_r}\otimes B^{1,\mu_{r-1}}\otimes \cdots \otimes B^{1,\mu_1}.$ In the following discussion, we view $B(\mu)$ as a $U_{q}(\mathfrak{g}_0)$-crystal, unless otherwise stated.
We denote by $\mathrm{W}(\mu)$ the set of $n$-tuple of rows $(w_n,\ldots,w_1)$ such that $\mathrm{wt}(w_n*\cdots*w_1)=\mu.$ For $\BBw=(w_n,\ldots,w_1)\in\mathrm{W}(\mu),$  define the weight $\mathrm{wt}(\BBw)$ as $\sum_{i}m_i\epsilon_i\in\overline{P},$  where $m_i$($1\leq i\leq n$) is the length of the row $w_i.$ For a given $\BBw\in\mathrm{W}(\mu),$ there exists a unique element $b_{\BBw}=b_r\otimes \cdots\otimes b_1\in B(\mu)$ determined by the condition:
\begin{align}
&\text{for any } i,k\in\mathbb{Z}_{>0}, \text{ the letter } i \text{ appears in } b_k
\text{ if and only if the letter } k \\ &\text{appears in } w_i.\nonumber
\end{align}
We define a map $\Psi:\mathrm{W}(\mu)\rightarrow B(\mu)$ by $\BBw\mapsto b_{\BBw}.$ It is easy to see that $\Psi$ gives a bijection, and preserves the weight, namely, we have
\begin{equation}\label{3.4}
\mathrm{wt}(b_{\BBw})=\mathrm{wt}(\BBw).
\end{equation}

\begin{exa}
\emph{Assume that} $\mu=(3,2,2,1)$ \emph{and }$n=5$\emph{.} \emph{Let} $\boldsymbol{w}=(-,4,1,113,223)\in \mathrm{W}(\mu).$  \emph{Then}
$b_{\boldsymbol{w}}=4\otimes12\otimes11\otimes223.$
\end{exa}

Let $(w_{2},w_{1})$ be a pair of rows with $w_{1}=x_{p}x_{p-1}\cdots x_{1}$ and $%
w_{2}=y_{q}y_{q-1}\cdots y_{1}$. For $1\leq t\leq q,$ we define an array $T(t)$ by
\begin{equation}
T(t)=
\begin{array}{llllllll}
&  & x_{p} & x_{n-1} & \cdots  & \cdots  & \cdots  & x_{1} \\
y_{q} & \cdots  & y_{t} & y_{t-1} & \cdots  & y_{1} &  &
\end{array},
\end{equation} and set $$T(0)=%
\begin{array}{llllll}
&  &  & x_{p} & \cdots  & x_{1} \\
y_{q} & \cdots  & y_{1} &  &  &
\end{array}.%
$$
Let $t_{0}$ be the largest
number such that $T(t_{0})$ is a tableau. If $w_{i}$ is the empty word ($i=1$ or $2$), we set $t_{0}=0$ and regard $T(t_0)$ as a tableau with two rows, whose $i$-th row is empty. Put $T_{(w_{2},w_{1})}=T(t_{0}).$

\begin{exa}
$w_{2}=122,w_{1}=114$, \emph{we have}
\begin{align*}
&T(0)=\begin{array}{llllll}
&  &  & 1 & 1 & 4 \\
1 & 2 & 2 &  &  &
\end{array},
&T(1)=%
\begin{array}{lllll}
&  & 1 & 1 & 4 \\
1 & 2 & 2 &  &
\end{array}\\
&T(2)=%
\begin{array}{llll}
& 1 & 1 & 4 \\
1 & 2 & 2 &
\end{array},
&T(3)=%
\begin{array}{lll}
1 & 1 & 4 \\
1 & 2 & 2%
\end{array}.
\end{align*}
\emph{In this case}, $t_{0}=2$ \emph{and} $T_{(w_2,w_1)}=T(2)=%
\begin{array}{llll}
& 1 & 1 & 4 \\
1 & 2 & 2 &
\end{array}.
$
\end{exa}

If there exists an upper jeu de taquin position $e$ in the first row of $T_{(w_{2},w_{1})}$  and assume that $w'_i$ ($i=1,2$) is the $i$-th row of $\mathrm{jdt}_e(T_{(w_2,w_1)}),$ we set $\mathrm{e}(w_2,w_1)=(w'_2,w'_1).$ If there exists no upper jeu de taquin position in the first row, write $\mathrm{e}(w_2,w_1)=0.$
Similarly, if there exists a lower jeu de taquin position $f$ in the first row of $T_{(w_{2},w_{1})}$  and assume that $w''_i$ ($i=1,2$) is the $i$-th row of $\mathrm{jdt}_f(T_{(w_2,w_1)}),$ we set $\mathrm{f}(w_2,w_1)=(w''_2,w''_1).$ If there exists no lower jeu de taquin position in the first row, write $\mathrm{f}(w_2,w_1)=0.$

\begin{Def}
\emph{For $i\in I_0,$ we define the Kashiwara operators $e_i,f_i:\mathrm{W}(\mu)\rightarrow \mathrm{W}(\mu)\sqcup\{0\}$  as follows. Let $\boldsymbol{w}=(w_n,\ldots,w_1)\in \mathrm{W}(\mu).$}
\begin{enumerate}
\item[(1)]
\emph{If $\mathrm{e}(w_{i+1},w_i)=(w'_{i+1},w'_i)\neq 0$,
define $e_i(\boldsymbol{w})=(w'_n,\ldots,w'_1)$, where $w'_k=w_k$ for $k\neq i,i+1$; if $\mathrm{e}(w_{i+1},w_i)=0$, define $e_i(\boldsymbol{w})=0.$}
\item[(2)]
\emph{If $\mathrm{f}(w_{i+1},w_i)=(w'_{i+1},w'_i)\neq 0$,
define $f_i(\boldsymbol{w})=(w'_n,\ldots,w'_1)$, where $w'_k=w_k$ for $k\neq i,i+1$; if $\mathrm{f}(w_{a+1},w_a)=0$, define $f_i(\boldsymbol{w})=0.$}
\end{enumerate}
\end{Def}

\begin{prop}\label{thm 3.3.3}
 The map $\Psi$ commutes with the Kashiwara operators, namely,
\begin{equation}\label{F 3.11}
e_ib_{\BBw}=b_{e_i(\BBw)} \text{   and   } f_ib_{\BBw}=b_{f_i(\BBw)}
\end{equation}
for $i\in I_0$ and $\BBw\in \BW(\mu)$.
\end{prop}
\begin{proof} We only prove the first assertion of (\ref{F 3.11}). The second one can be proved similarly.
For $b=b_r\otimes \cdots b_1$ with $b_k=x_{1,k}\cdots x_{\mu_k,k},$ $1\leq k\leq r,$ let
\begin{equation}
\widetilde{b}=x_{1,r}\otimes \cdots x_{\mu_r,r}\otimes\cdots \otimes x_{1,1}\otimes \cdots\otimes x_{\mu_1,1 }\in (B^{1,1})^{\otimes N}.
\end{equation}
 Then the map $b\mapsto \widetilde{b}$ gives a crystal embedding $B(\mu)\hookrightarrow (B^{1,1})^{\otimes N}.$
In order to prove (\ref{F 3.11}), it is enough to show
\begin{equation}\label{F 3.12(1)}
e_i\widetilde{b}_{\BBw}=\widetilde{b}_{e_i(\BBw)},
\end{equation} for $i\in I_0$ and $\BBw\in \mathrm{W}(\mu).$

We show (\ref{F 3.12(1)}). For $\BBw=(w_n,\ldots,w_1)\in \mathrm{W}(\mu)$ with $w_k=a_{\ell_k,k}a_{\ell_k-1,k}\cdots a_{1,k},$ $1\leq k\leq n,$ let
$S_{w_k}=\{(x,k)\mid 1\leq x\leq \ell_k\}.$ Set $$S_{\BBw}=\coprod_{1\leq k\leq n}S_{w_k}.$$ Define a total order $<$ on $S_{\BBw}$ by
\begin{enumerate}
\item[(1)]
 $(x_1,k)<(x_2,k)$ if $x_1<x_2;$

\item[(2)]
 $(x_1,k_1)< (x_2,k_2)$ if $k_1< k_2$ and $a_{x_1,k_1}=a_{x_2,k_2};$

\item[(3)]
 $(x_1,k_1)< (x_2,k_2)$ if $k_1\neq k_2$ and $a_{x_1,k_1}>a_{x_2,k_2}.$
\end{enumerate}
We express the elements in $S_{\BBw}$ as $c_1<c_2<\cdots<c_N$ along this total order. Then it is easy to check that $\widetilde{b}_{\BBw}$ can be written as
\begin{equation}\label{3.8}
\widetilde{b}_{\BBw}=y_{c_1}\otimes y_{c_2} \otimes \cdots \otimes y_{c_N},
\end{equation}
 where $y_{c_j}=k$ if $c_j=(x,k)\in S_{\BBw}.$
We fix $i$ and $i+1$, and write $x^+=(x,i)$ for $1\leq x\leq p$ and  $y^-=(y,i+1)$ for $1\leq y\leq q,$ where $p=\ell_i,$ $q=\ell_{i+1}.$ Write $a_{x^+}=a_{x,i}$ and $b_{y^-}=b_{y,i+1}.$
Then the restriction of the order on the set $$S_{\BBw,i}\sqcup S_{\BBw,i+1}=\{1^+,2^+,\ldots,p^+,1^-,2^-,\ldots,q^-\}$$ satisfies the following properties
\begin{align*}
x^+>y^+ &\Longleftrightarrow x>y,\\
x^->y^- &\Longleftrightarrow x>y,\\
x^+>y^- &\Longleftrightarrow a_{x^+}<a_{y^-},\\
y^->x^+ &\Longleftrightarrow a_{x^+}\geq a_{y^-}.
\end{align*}
Note, for $x\in B^{1,1}$, that the $i$-signature is $+$  (resp. $-$, empty) if $x=i$ (resp. $x=i+1$, otherwise). Hence it is clear that the $i$-signature of $\widetilde{b}$ is given by
\begin{equation}\label{F 3.13(1)}
\varepsilon_{k_1}\varepsilon_{k_2}\ldots \varepsilon_{k_{p+q}},
\end{equation} where if we denote the elements in $S_{w_i}\sqcup S_{w_{i+1}}$ as $k_1<k_2<\cdots<k_{p+q}$ following the total order, then $\varepsilon_{k_j}$ is determined as
\begin{equation*}
\varepsilon_{k_j} = \begin{cases}
                +   &\quad\text{ if $k_j\in S_{\BBw,i}$}, \\
                -        &\quad\text{ if $k_j\in S_{\BBw,i+1}$. }
        \end{cases}
\end{equation*}
We call (\ref{F 3.13(1)}) the signature associated to the tableau $T_{(w_{i+1},w_i)}.$

First we consider a special case.

\begin{enumerate}

\item[(a)]
If $T_{(w_{i+1},w_i)}$ is of the form
$$\begin{array}{llllll}
  &  & a_{p^+} & a_{(p-1)^+} & \cdots & a_{1^+} \\
a_{q^-} & \cdots & a_{p^-} & a_{(p-1)^-} & \cdots & a_{1^-}
\end{array} ,$$ then the reduced $i$-signature of $\widetilde{b}_{\BBw}$ has no $+$.

\end{enumerate}
In fact, since $a_{j^-}>a_{j^+},$ we have $j^+>j^-$ for each $1\leq j\leq p$ and all  $+$ in the $i$-signature appear in the pairs $(\varepsilon_{i^-},\varepsilon_{i^+}),$ hence all the $+$ in the $i$-signature can be removed.
Similarly, we have

\begin{enumerate}

\item[(b)]
 If $T_{(w_{i+1},w_i)}$ is of the form
$$\begin{array}{llllll}
a_{p^+} & a_{(p-1)^+} & \cdots & a_{(p-q+1)^+} & \cdots & a_{1^+}\\
a_{q^-} & a_{(q-1)^-} & \cdots & a_{1^-}
\end{array},$$ then the reduced $i$-signature has no $-.$

\end{enumerate}
In the general case, assume that $T_{(w_{i+1},w_i)}$ is written as $$T_{(w_{i+1},w_i)}=
\begin{array}{llllllllll}
  &   &   &  a_{p^+} & a_{(p-1)^+} & \cdots & a_{(p-m+1)^+} & a_{(p-m)^+} & \cdots & a_{1^+} \\
a_{q^-} & a_{(q-1)^-} & \cdots & a_{m^-} & a_{(m-1)^-} & \cdots & a_{1^-}
\end{array}.$$
Set $J=\{t\mid 1\leq t\leq m, \text{ } a_{t^-}\leq a_{(p-(m-t+1))^+}\}.$ We consider the case $J\neq\emptyset$ or $J=\emptyset$ separately. \

\par\smallskip\noindent\textbf{Case 1.}  $J\neq\emptyset.$

Let $t_0$ be the largest number contained in $J.$ We consider the procedure of the jeu de taquin slide of $T_{(w_{i+1},w_i)}$ into the upper jeu de taquin position of the first row. Then the last step before the pawn $\clubsuit$ going down is of the form
$$\begin{array}{lllllllllll}
   &   & a_{p^+} & a_{(p-1)^+} & \cdots & a_{(p-(m-t_0))^+} & \clubsuit & a_{(p-(m-t_0+1))^+} & \cdots & \cdots & a_{1^+}\\
a_{q^-} & \cdots & a_{(m+1)^-} & a_{m^-} & \cdots & a_{(t_0+1)^-} & a_{t^{-}_{0}} & a_{(t_0-1)^-} & \cdots & a_{1^-}
\end{array}.$$
We divide the $i$-signature of $\widetilde{b}_{\BBw}$ into three parts
$$\underset{\text{part 1}}{\underbrace{\varepsilon_{k_1}\cdots \cdots\varepsilon_{k_s}}} \varepsilon_{t_{0}^{-}}      \underset{\text{part 2}}{\underbrace{\varepsilon_{k_{s+2}}\cdots \cdots\varepsilon_{k_{p+q}}}}.$$
Note that for $p-(m-t_0)\leq j\leq p,$ we have $j^+>t_0^-$ since $a_{j^+}<a_{((m+1)-(p-j))^-}\leq a_{t_0^-}.$ Also for $1\leq j\leq p-(m-t_0+1),$ we have $j^+<t_0^-$ since $a_{j^+}\geq a_{(p-m+t_0-1)^+}\geq a_{t_0^-}.$
It is easy to see that $$\{k_1,k_2,\ldots,k_s\}=\{1^+,2^+,\ldots,(p-(m-t_0+1))^+, 1^-,2^-,\ldots,(t_0-1)^-\},$$ and $\varepsilon_{k_1}\cdots \varepsilon_{k_s}$ is the signature associated to the tableau
$$\begin{array}{llll}
a_{(p-(m-t_0+1))^+} & \cdots & \cdots & a_{1^+}\\
a_{(t_0-1)^-} & \cdots & a_{1^-}
\end{array}.$$
Similarly, $\varepsilon_{k_{s+2}}\cdots \cdots\varepsilon_{k_{p+q}}$ is the signature associated to the tableau
$$
\begin{array}{llllll}
    &    & a_{p^+} & a_{(p-1)^+} & \cdots & a_{(p-(m-t_0))^+}\\
a_{q^-} & \cdots & a_{(m+1)^-} & a_{m^-} & \cdots & a_{(t_0+1)^-}
\end{array}.$$
By (a) and (b), we know that there exists no $-$ in the reduced signature of part 1 and no $+$ in the reduced signature of part 2. Hence, $\varepsilon_{t_{0}^{-}}$ is the left most $-$ in the reduced $i$-signature of $\widetilde{b}_{\BBw}$. This implies that
\begin{equation}\label{nF 3.4}
e_i\widetilde{b}_{\BBw}=y_{c_1}\otimes y_{c_2} \otimes \cdots \otimes e_iy_{(t_0,i+1)}\otimes\cdots\otimes y_{c_N}.
\end{equation}

We write $e_i(\BBw)$ as $e_i(\BBw)=(w'_n,\cdots,w'_1)$ with $w'_k=b_{\ell'_k,k}b_{\ell'_k-1,k}\cdots b_{1,k}.$ Then $S_{w'_k}=\{(x,k)\mid 1\leq x\leq \ell'_k\}$ is determined as before.
It is easy to check that
\begin{equation}
\ell'_k = \begin{cases}
                \ell_k+1       &\quad\text{ if $k=i$}, \\
                \ell_k-1        &\quad\text{ if $k=i+1$. }\\
                \ell_k            &\quad\text{ otherwise, }
        \end{cases}
\end{equation}
and
\begin{equation}\label{nF 3.5}
b_{x,k} = \begin{cases}
                a_{t_0^-}       &\quad\text{ if $k=i, x=p-(m-t_0)$}; \\
                a_{(x-1)^+}          &\quad\text{ if $k=i, p-(m-t_0)\leq x\leq \ell'_{i}; $ }\\
                a_{(x+1)^-}            &\quad\text{ if $k=i+1,t_0\leq x\leq \ell'_{i+1}; $}\\
                a_{x,k}                     &\quad\text{ otherwise.}
        \end{cases}
\end{equation}
As in (\ref{3.8}), $\widetilde{b}_{e_i(\BBw)}$ can be written as
$$\widetilde{b}_{e_i\BBw}=y'_{c'_1}\otimes y'_{c'_2} \otimes \cdots \otimes y'_{c'_N},$$ where $c'_j\in S_{e_i\BBw}$, $c'_1<c'_2<\cdots<c'_N.$
Here $y'_{c'_j}=k$ if $c'_j=(x,k)\in S_{e_i(\BBw)}$ when we write the elements in $S_{e_i(\BBw)}$ as $c'_1<c'_2<\cdots<c'_N.$
Define a map $h:S_{\BBw}\rightarrow S_{e_i(\BBw)}$ by
\begin{equation*}
h(x,k) = \begin{cases}
                (x+1,i)     &\quad\text{ if $k=i, p-(m-t_0)\leq x\leq \ell_i$}; \\
                (p-(m-t_0),i)          &\quad\text{ if $k=i+1, x=t_0; $ }\\
                (x-1,i+1)           &\quad\text{ if $k=i+1,t_0+1\leq x\leq \ell_{i+1}; $}\\
                (x,k)                    &\quad\text{ otherwise.}
        \end{cases}
\end{equation*}
By a direct computation using (\ref{nF 3.5}), we see that $h$ is the unique bijection such that $h(c_j)<h(c_k)$ for $1\leq j<k\leq N.$
Hence we have
\begin{eqnarray*}
\widetilde{b}_{e_i\BBw} &=& y'_{c'_1}\otimes y'_{c'_2} \otimes \cdots \otimes y'_{c'_N} \\
&=& y'_{h(c_1)}\otimes y'_{h(c_2)} \otimes \cdots \otimes y'_{h(c_N)}.
\end{eqnarray*}
Noticing that
\begin{equation*}
y'_{h(c_i)} = \begin{cases}
                y_{c_i}   &\quad\text{ if $c_i\neq t_0^-$}, \\
                y_{c_i}-1        &\quad\text{ if $c_i=t_0^-$, }
        \end{cases}
\end{equation*}
we obtain
\begin{equation}\label{nF 3.6}
\widetilde{b}_{e_i(\BBw)}=y_{c_1}\otimes y_{c_2} \otimes \cdots \otimes (y_{(t_0,i+1)}-1)\otimes\cdots\otimes y_{c_N}.
\end{equation}
Since $e_iy_{(t_0,i+1)}=y_{(t_0,i+1)}-1,$ we obtain (\ref{F 3.12(1)}), from (\ref{nF 3.4}) and (\ref{nF 3.6}).

\par\smallskip\noindent\textbf{Case 2}  $J=\emptyset.$

In this case, $T_{(w_{i+1},w_i)}$ must be of the form
$$\begin{array}{llllll}
   &   &   & a_{p^+} & \cdots & a_{1^+}\\
a_{q^-} & \cdots & a_{(p+1)^-} & a_{p^-} & \cdots & a_{1^-}
\end{array}.$$
Then the jeu de taquin operation can be written as
$$\begin{array}{lllllll}
    &    & a_{p^+} & a_{(p-1)^+} & \cdots & a_{1^+} & \clubsuit\\
a_{q^-} & \cdots & a_{(p+1)^-} & a_{p^-} & \cdots & a_{2^-} & a_{1^-}
\end{array}.$$ It is clear that $1^-$ is the smallest element in $S_{\BBw,i}\sqcup S_{\BBw,i+1}$. Divide the $i$-signature into two parts
$$\varepsilon_{1^{-}}      \underset{\text{part 2}}{\underbrace{\varepsilon_{k_2}\cdots \cdots\varepsilon_{k_{p+q}}}}.$$ Note that part 2 is the signature associated to the tableau
$$\begin{array}{llllll}
    &    & a_{p^+} & a_{(p-1)^+} & \cdots & a_{1^+}\\
a_{q^-} & \cdots & a_{(p+1)^-} & a_{p^-} & \cdots & a_{2^-}
\end{array}.$$ By (a), there exists no $+$ in the reduced signature of part 2, thus $\varepsilon_{1^-}$ is the left most $-$ in the reduced $i$-signature of $\widetilde{b}_{\BBw}.$ By a similar argument as in the case 1, we obtain (\ref{F 3.12(1)}) also in this case. Thus the proposition is proved.
\end{proof}

For $i\in I_0,$ $\BBw\in \mathrm{W}(\mu)$, we define $\varepsilon_i(\BBw), \varphi_i(\BBw)\in \mathbb{Z}_{\geq 0}$ by
\begin{align}
\varepsilon_i(\BBw) &=\max\{k\mid e^{k}_i\BBw\neq 0\}\label{F 3.13}\\
\varphi_i(\BBw) &=\max\{k\mid f^{k}_i\BBw\neq 0\}.\label{F 3.14}
\end{align}
Then we have the following result as a corollary of Proposition \ref{thm 3.3.3}.
\begin{thm}\label{cor 3.3.4}
The maps $e_i,f_i: \BW(\mu) \rightarrow \BW(\mu)\sqcup\{0\}$, $\varepsilon_i,\varphi_i:\BW(\mu)\rightarrow \mathbb{Z}_{\geq 0}$ ($i\in I_0$) and $\mathrm{wt}: \BW(\mu)\rightarrow \overline{P}$ define a $U_{q}(\mathfrak{g}_0)$-crystal structure on $\BW(\mu)$. Moreover, the bijection $\Psi$ gives a $U_{q}(\mathfrak{g}_0)$-crystal isomorphism.
\end{thm}
\begin{proof}
Let $\Psi: \mathrm{W}(\mu)\rightarrow B(\mu)$ be the bijection given in 3.4. In order to prove the theorem, it is enough to show that $e_i,f_i,\varepsilon_i,\varphi_i$ and $\mathrm{wt}$ commute with $\Psi.$ By (\ref{3.4}), $\mathrm{wt}$ commutes with $\Psi.$ By Proposition \ref{thm 3.3.3}, $e_i$ and $f_i$ commute with $\Psi.$ Then by (\ref{F 3.13}) and (\ref{F 3.14}), $\varepsilon_i$ and $\varphi_i$ commute with $\Psi.$ The theorem is proved.
\end{proof}

\subsection{}Let $\lambda,$ $\rho$ be partitions with $l(\lambda)\leq n$ such that $\lambda\supset\rho$ and $|\lambda|-|\rho|=N.$ Then there exists a natural embedding $\mathrm{Tab}(\lambda-\rho,\mu)\hookrightarrow \mathrm{W}(\mu),$ sending the tableau $T$ to the element $\BBw_{T}=(w_n,\ldots,w_1),$ where $w_i$ is the $i$-th row of $T.$ We write $b_T=\Psi(\BBw_T)\in B(\mu)$ for $T\in \mathrm{Tab}(\lambda-\rho,\mu).$

\begin{lem}\label{lem20151125}Assume that the vertical slide of the form
\begin{equation*}\begin{array}{llllllllllllll}
  &   &   & a_1 & a_2 & \cdots & a_p & \clubsuit & a_{p+1} & \cdots & a_{n-t} & \cdots & a_m\\
  & b_1 & \cdots & b_t & b_{t+1} & \cdots & b_{t+p-1} & b_{t+p} & b_{t+p+1} & \cdots & b_n\\
\end{array}\rightarrow
\end{equation*}
$$\begin{array}{llllllllllllll}
  &   &   & a_1 & a_2 & \cdots & a_p & b_{t+p} & a_{p+1} & \cdots & a_{n-t} & \cdots & a_m\\
  & b_1 & \cdots & b_t & b_{t+1} & \cdots & b_{t+p-1} & \clubsuit & b_{t+p+1} & \cdots & b_n\\
\end{array}$$
occurs in between the $i$-th row and the $(i+1)$-th row in a jeu de taquin operation for a tableau. Put
$$\left\{\begin{array}{ll}
  w_1=a_1\cdots a_m,$ $w_2=b_1\cdots b_n,\\
  w_2=b_1\cdots b_n,
\end{array}
\right. \text{ and } \left\{\begin{array}{ll}
  w'_1=a_1\cdots a_pb_{t+p}a_{p+1}\cdots a_m,\\
  w'_2=b_1\cdots b_{t+p-1}b_{t+p+1}\cdots b_n.
\end{array}
\right.$$
 Then $\mathrm{e}(w_2,w_1)=(w'_2,w'_1).$
\end{lem}
\begin{proof}
First we note the following.

($\dag$) $T_{(w_2,w_1)}$ is given by the two-row tableau
$$\begin{array}{llllllllllllll}
   &   &  & a_1 & \cdots & a_{p-1} & a_p & a_{p+1} & \cdots & a_{n-t} & \cdots & a_m\\
 b_1 & \cdots & b_t & b_{t+1} & \cdots & b_{t+p-1} & b_{t+p} & b_{t+p+1} & \cdots & b_n
\end{array}.$$
In fact, since $\begin{array}{llllllllllllll}
a_1 & a_2 & \cdots & a_p\\
b_t & b_{t+1} & \cdots & b_{t+p-1}
\end{array}$ is a tableau, we have
\begin{equation}\label{con.0.2}
b_{t+k}\geq b_{t+k-1} >a_k \text{ for } 1\leq k\leq p.
\end{equation} Hence $$\begin{array}{llllllllllllll}
   &   &  & a_1 & \cdots & a_{p-1} & a_p & a_{p+1} & \cdots & a_{n-t} & \cdots & a_m\\
 b_1 & \cdots & b_t & b_{t+1} & \cdots & b_{t+p-1} & b_{t+p} & b_{t+p+1} & \cdots & b_n
\end{array}$$ is a tableau. Since the pawn slides vertically by our assumption, we have $b_{t+p}\leq a_{p+1}.$ Hence, $$\begin{array}{llllllllllllll}
   &   & a_1 & a_2 & \cdots & a_{p} & a_{p+1} & a_{p+2} & \cdots & a_{n-t+1} & \cdots & a_m\\
 b_1 & \cdots & b_t & b_{t+1} & \cdots & b_{t+p-1} & b_{t+p} & b_{t+p+1} & \cdots & b_n
\end{array}$$ is not a tableau. By the definition of $T_{(w_2,w_1)}$, ($\dag$) holds.

Consider the jeu de taquin slide of $T_{w_2,w_1}$ as follows.
$$\begin{array}{llllllllllllll}
   &   & \clubsuit & a_1 & \cdots & a_{p-1} & a_p & a_{p+1} & \cdots & a_{n-t} & \cdots & a_m\\
 b_1 & \cdots & b_t & b_{t+1} & \cdots & b_{t+p-1} & b_{t+p} & b_{t+p+1} & \cdots & b_n
\end{array}$$ By the condition (\ref{con.0.2}), the pawn slides horizontally up to the right of $a_p.$
Then by the condition $b_{t+p}\leq a_{p+1}$ as before, the pawn slides vertically, thus we obtain
$$\begin{array}{llllllllllllll}
  &   &   & a_1 & a_2 & \cdots & a_p & b_{t+p} & a_{p+1} & \cdots & a_{n-t} & \cdots & a_m\\
  & b_1 & \cdots & b_t & b_{t+1} & \cdots & b_{t+p-1} & \clubsuit & b_{t+p+1} & \cdots & b_n
\end{array}.$$ After that by the horizontal slides of the pawn, we reach the final position
$$\begin{array}{llllllllllllll}
  &   &   & a_1 & a_2 & \cdots & a_p & b_{t+p} & a_{p+1} & \cdots & a_{n-t-1} & a_{n-t} & \cdots & a_m\\
  & b_1 & \cdots & b_t & b_{t+1} & \cdots & b_{t+p-1} & b_{t+p+1} & b_{t+p+2} & \cdots & b_n & \clubsuit
\end{array}.$$
This shows that $\mathrm{e}(w_2,w_1)=(w'_2,w'_1).$ The lemma is proved.
\end{proof}
We have the following proposition.
\begin{prop}\label{cor 3.3.5}
Let $\lambda,$ $\rho$ be partitions with $l(\lambda)\leq n$ such that $\lambda\supset\rho$ and $|\lambda|-|\rho|=N.$  For $T\in \mathrm{Tab}(\lambda-\rho,\mu),$ let  $e$ (resp. $f$) be the upper (resp. lower) jeu de taquin position of the $i$-th (resp.$j$-th) row of $T$. Assume that the pawn $\clubsuit$(resp.$\spadesuit$) stops at the $j$-th ($i$-th) row after the process of jeu de taquin operation for some $j\leq i.$ Then
\begin{equation}\label{F 3.16}
b_{\mathrm{jdt}_{e}(T)}=e_{j-1}\ldots e_{i+1}e_ib_T,
\end{equation}
\begin{equation}\label{F 3.17}
b_{\mathrm{jdt}_{f}(T)}=f_{i}\ldots f_jf_{j-1}b_T.
\end{equation} In particular, $b_T$ is contained in the $U_{q}(\mathfrak{g}_0)$-connected component generated by the maximal vector $b_{\mathrm{jdt}(T)}.$
\end{prop}
\begin{proof}
We only prove (\ref{F 3.16}). (\ref{F 3.17}) can be proved similarly. The horizontal slide in the jeu de taquin operation does not give any change for the elements in $\mathrm{W}(\mu).$ By applying Lemma \ref{lem20151125} for each vertical slide in the jeu de taquin operation, we have
\begin{equation}\label{3.7.4}
\boldsymbol{w}_{\mathrm{jdt}_e(T)}=e_{j-1}\ldots e_i\boldsymbol{w}_T.
\end{equation} Now, by applying the crystal isomorphism $\Psi$ on the both sides of (\ref{3.7.4}) we obtain
$$b_{\mathrm{jdt}_e(T)}=\psi(\boldsymbol{w}_{\mathrm{jdt}_e(T)})=\psi(e_{j-1}\ldots e_i\boldsymbol{w}_T)=e_{j-1}\ldots e_i\psi(\boldsymbol{w}_T)=e_{j-1}\ldots e_ib_T.$$ Thus (\ref{F 3.16}) is proved.
\end{proof}

\subsection{} Let $B$ be a $U_{q}(\mathfrak{g}_0)$-crystal. For $\lambda\in \mathscr{P}_N$ such that $l(\lambda)\leq n,$ we define a subset $P(B,\lambda)$ of $B$ by
\begin{equation}
P(B,\lambda)=\{b\in B\mid\mathrm{wt}(b)=\lambda, e_ib=0 \text{ for any } i\in I_0 \}.
\end{equation}
 It is well-known (see, e.g., [NY]) that the map
\begin{eqnarray}\label{3.6.2}
 \Psi_\lambda:\mathrm{Tab}(\lambda,\mu) &\isom &P(B(\mu),\lambda)\\
 T &\mapsto &b_T\nonumber
\end{eqnarray} gives a bijection.
Generalizing the definition of $P(B(\mu),\lambda),$ we shall define a subset $P(B(\mu),\Bla)$ of $B(\mu)$ associated to a double partition $\Bla$ as follows; let $\Bla=(\lambda',\lambda'')\in\mathscr{P}_{N,2}$ with $\lambda'=(\lambda'_1,\ldots,\lambda'_s),$ $\lambda''=(\lambda''_1,\ldots,\lambda''_t),$ such that $s+t\leq n.$ We regard $\Bla$ as an element in $\overline{P}$ as in 3.2. Set
\begin{equation}
P(B,\Bla)=\{b\in B\mid\mathrm{wt}(b)=\Bla, e_ib=0 \text{ for any } i\in I_0-\{ s\} \}.
\end{equation}
Under the identification $\mathrm{Tab}(\Bla,\mu)$ with $\mathrm{Tab}(\xi_{\Bla,a},\mu)$ with $a=\lambda_s,$ we define a map
\begin{eqnarray}
\Psi_{\Bla}:\mathrm{Tab}(\Bla,\mu) &\rightarrow &B(\mu)\\
T &\mapsto &b_T.\nonumber
\end{eqnarray}

\begin{prop}\label{cor 3.4.2}
Under the notation above, the map $\Psi_{\Bla}$ gives a bijection.
\end{prop}
\begin{proof}
By Theorem \ref{cor 3.3.4}, one can write $P(B(\mu),\Bla)$ as $$P(B(\mu),\Bla)=\{b_{\BBw}\mid \BBw\in B(\mu), \mathrm{wt}(\BBw)=\Bla, \text{ and } e_i\BBw=0 \text{ for } i\in I_0-\{ s\}\}.$$
Let $\BBw=(w_n,\ldots,w_1)\in \mathrm{W}(\mu)$ with $w_k=a_{\ell_k,k}a_{\ell_k-1,k}\cdots a_{1,k},$ $1\leq k\leq n.$ As in the proof of Proposition \ref{cor 3.3.5} , for $i\in I_0,$ the condition $e_i(\BBw)=0$ is equivalent to the condition that $\ell_i\geq \ell_{i+1}$ and the word $w_{i+1}*w_i$ is compatible with the partition $(\ell_i,\ell_{i+1}).$
On the other hand $\mathrm{wt}(\BBw)=\Bla$ if and only if
\begin{equation*}
\ell_k = \begin{cases}
                \lambda'_k   &\quad\text{ if $1\leq k\leq s$}, \\
                \lambda''_{k-s}       &\quad\text{ if $s+1\leq k\leq n.$ }
        \end{cases}
\end{equation*}
Hence, $\mathrm{wt}(\BBw)=\Bla,$ $e_i\BBw=0$ for $i\in I_0-\{s\}$ if and only if $w_s*\cdots *w_1$ is compatible with $\lambda'$ and $w_{t+s}*\cdots* w_{s+1}$ is compatible with $\lambda''.$ The proposition follows from this.
\end{proof}

\subsection{} It is known that there exists the energy function $\mathrm{E}:B(\mu)\rightarrow \mathbb{Z}_{\geq 0}$ (for details, see, e.g., [NY]). The following result was proved by  Nakayashiki-Yamada [NY].
\begin{thm}[{[NY]}]\label{thm 3.5.1}
 Let $\lambda,\mu\in\mathscr{P}_N$ such that $l(\lambda)\leq n.$
\begin{enumerate}
\item[(1)]
For $T\in \mathrm{Tab}(\lambda,\mu),$ we have
\begin{equation}\label{F 3.20}
\mathrm{E}(b_T)=\mathrm{c}(T).
\end{equation}

\item[(2)]
 The energy function $\mathrm{E}:B(\mu)\rightarrow \mathbb{Z}_{\geq 0}$ is constant on each $U_{q}(\mathfrak{g}_0)$-connected component of $B(\mu)$.
\end{enumerate}
\end{thm}

By summing up the previous results, we obtain the following theorem, which is a generalization of Theorem \ref{thm 3.5.1} (1).

\begin{thm}\label{Thm 3.5.2}
Let $\lambda,$ $\rho$ be partitions such that $l(\lambda)\leq n,$ and $\lambda\supset\rho$ with $|\lambda|-|\rho|=N.$ Let $\mu\in\mathscr{P}_N.$ Then for $T\in \mathrm{Tab}(\lambda-\rho,\mu),$ we have
\begin{equation}\label{3.22}
\mathrm{E}(b_T)=\mathrm{c}(T).
\end{equation}   In particular, let $\Bla=(\lambda',\lambda'')\in\mathscr{P}_{N,2}$ with $\lambda'=(\lambda'_1,\ldots,\lambda'_s),$ $\lambda''=(\lambda''_1,\ldots,\lambda''_t)$ such that $s+t\leq n.$ Then for $T\in\mathrm{Tab}(\Bla,\mu),$ we have
\begin{equation}\label{3.23}
\mathrm{E}(b_T)=\mathrm{c}(T).
\end{equation}
\end{thm}
\begin{proof}
Let $T\in \mathrm{Tab}(\lambda-\rho,\mu).$ By Proposition \ref{cor 3.3.5}, $b_T$ is contained in the $U_{q}(\mathfrak{g}_0)$-connected component of $B(\mu)$ generated by the maximal vector $b_{\mathrm{jdt}(T)}.$ Since $\mathrm{E}$ is constant on each $U_{q}(\mathfrak{g}_0)$-connected component, we have $\mathrm{E}(b_T)=\mathrm{E}(b_{\mathrm{jdt}(T)}).$ Since $\mathrm{jdt}(T)\in \mathrm{Tab}(\nu,\mu)$ for some $\nu\in\mathscr{P}_{N},$ $\mathrm{E}(b_T)=\mathrm{E}(b_{\mathrm{jdt}(T)})$ by Theorem \ref{thm 3.5.1} (1). Moreover, we have $\mathrm{c}(T)=\mathrm{c}(\mathrm{jdt}(T))$ by Corollary 1.4.5. Thus (\ref{3.22}) holds. (\ref{3.23}) is a special case of (\ref{3.22}). The theorem is proved.
\end{proof}
For $\BBw=(w_n,\ldots,w_1)\in \mathrm{W}(\mu),$ we define the charge of $\BBw$ as
$$\mathrm{c}(\BBw)=\mathrm{c}(w_n*\cdots*w_1).$$
\begin{cor}
For $\BBw\in\mathrm{W}(\mu),$ $\mathrm{c}(\BBw)=\mathrm{E}(b_{\BBw}).$
\end{cor}
\begin{proof}
One can choose some tableau $T$ whose $i$-th row is $w_i$ for $1\leq i\leq n.$ Then $w(T)=w_n*\cdots*w_1,$ thus $\mathrm{c}(T)=\mathrm{c}(\BBw).$ By Theorem \ref{3.23}, $\mathrm{c}(T)=\mathrm{E}(b_T).$ We obtain $\mathrm{c}(\BBw)=\mathrm{E}(b_{\BBw})$ by noticing that $b_T=b_{\BBw}.$ The corollary is proved.
\end{proof}

We now obtain a $1D$ sum type expression of double Kostka polynomials in terms of crystals and the energy function.
\begin{thm}\label{thm 3.9.3}
Let $\Bla=(\lambda',\lambda''),\Bmu\in\mathscr{P}_{N,2 },$ with $\lambda'=(\lambda'_1,\ldots,\lambda'_s),$ $\lambda''=(\lambda''_1,\ldots,\lambda''_t)$ such that $s+t\leq n.$ Assume that $\Bmu=(-,\mu'').$  Then
\[
K_{\Bla,\Bmu}(t)=t^{|\lambda'|}\sum_{b\in P(B(\mu''),\Bla)}t^{2\mathrm{E}(b)}.
\]
\end{thm}
\begin{proof}
 The theorem is an immediate consequence of  Theorem \ref{nthm 2.3.2}, Proposition \ref{cor 3.4.2} and Theorem \ref{Thm 3.5.2}.
\end{proof}

\section{Rigged configurations}
\subsection{}Set $\mathscr{H}=I_0\times\mathbb{Z}_{>0}.$ For $B=B^{r_k,m_k}\otimes B^{r_{k-1},m_{k-1}}\otimes \cdots \otimes B^{r_1,m_1},$ $(r_i,m_i)\in \mathscr{H},$ the multiplicity array $L(B)$ is defined as $L(B)=(L^{(a)}_i)_{(a,i)\in \mathscr{H}},$ where $L^{(a)}_i$ denotes the number of $B^{a,i} $ within $B.$ In the case $B=B(\mu),$ let $L(\mu)=L(B).$
The rigged configurations are combinatorial objects depending on $L(B).$
Let $\lambda=(\lambda_1,\ldots,\lambda_k)$ be a partition. A rigged partition is a labeling $(\lambda_i,x_i)$ for each part $\lambda_i$ of $\lambda,$ where two labelings are identified if one is obtained from the other by permutations of $\{(x_i)\mid\lambda_i=j\}$ for fixed $j.$ We denote the rigged partition by $(\lambda,J),$ where $J=(x_i)_{1\leq i\leq k}$ is called the rigging of $\lambda.$ The pairs $(\lambda_i,x_i)$ are referred as strings; $\lambda_i$ is referred as the length of the string, and $x_i$ as the label of the string. The sequence of partitions $\nu=\{\nu^{(a)}\mid a\in I_0\}$ is called a configuration. A rigged configuration is a pair $(\nu,J)$ with configuration $\nu$ and a sequence $J=\{J^{(a)}\mid a\in I_0\},$ where $J^{(a)}$ is a rigging of the partition $\nu^{(a)}.$

\subsection{}In order to define a weight $\mathrm{wt}(\nu,J)$ and the vacancy numbers $p_i^{(a)}$ for a rigged configuration $(\nu,J),$ we need to give a multiplicity array $L=(L^{(a)}_i)_{(a,i)\in \mathscr{H}}$ satisfying the condition that $L^{(a)}_i=0$ for almost $(a,i)\in \mathscr{H}.$  A rigged configuration $(\nu,J)$ with the additional data $L$ is called a $L$-rigged configuration. The corresponding configuration $\nu$ is called a $L$-configuration.
Let $\lambda=(\lambda_1,\ldots,\lambda_k)$ be a partition. For an integer $i\geq0,$ set $m_i(\lambda)=\sharp\{j\mid\lambda_j=i\}$ and $Q_i(\lambda)=\sum_{j=1}^k\min(i,\lambda_j).$
Let $(\nu,J)$ be a $L$- configuration.  We write $m^{(a)}_i=m_i(\nu^{(a)})$ and $Q^{(a)}_i=Q_i(\nu^{(a)})$ for $I_0\cup\{0\}$.
 Define the weight $\mathrm{wt}(\nu,J)$ as
\begin{equation}\label{F 4.1.1}
\mathrm{wt}(\nu,J)=\sum_{(a,i)\in\mathscr{H}}iL_i^{(a)}\overline{\Lambda}_a-\sum_{(a,i)\in\mathscr{H}}im_i^{(a)}\alpha_a.
\end{equation}
We define the weight of the corresponding $L$-configuration $\nu$ as the weight of $(\nu,J).$ For $(a,i)\in\mathscr{H},$ we define the vacancy number $p_{i}^{(a)}$ as
\begin{equation}\label{F 4.1.2}
p_{i}^{(a)}=\sum_{j\geq 1}\min(i,j)L^{(a)}_j-\sum_{(b,j)\in\mathscr{H}}(\alpha_a,\alpha_b)\min(i,j)m_j^{(b)},
\end{equation} where $(\cdot,\cdot)$ is the normalized invariant form on the weight lattice $\overline{P}.$  For the convenience sake, we set $p^{(a)}_0=0.$ This convention fits to the formula in (\ref{F 4.1.2}).
A string $(i,x)\in(\nu^{(a)}, J^{(a)})$ is said to be singular if $x=p^{(a)}_{i}$.

 The following lemma is easily checked by using (\ref{F 4.1.2}).
\begin{lem}[{[Sc1, (3.4)]}]\label{Lem 4.1.1}
Let $(\nu,J)$ be a $L$-rigged configuration. Then
\begin{equation}\label{4.2.3}
2p_{i}^{(a)}-p_{i-1}^{(a)}-p_{i+1}^{(a)}\geq m_{i}^{(a-1)}-2m_{i}^{(a)}+m_{i}^{(a+1)}
\end{equation} for $i>0.$
In particular, the convexity condition
\begin{equation}\label{4.2.4}
p_{i}^{(a)}\geq \dfrac{1}{2}(p_{i-1}^{(a)}+p_{i+1}^{(a)})   \text{    if  }  m_{i}^{(a)}=0
\end{equation} holds.
\end{lem}

The following lemma is used in later discussions:
\begin{lem}\label{Lem 4.1.2}
Assume that $a_{1}, a_{2}, \ldots,a_{m}\in\mathbb{R}$ ($m\geq 3$) satisfy the condition $2a_{i}\geq a_{i-1}+a_{i+1}$ for $2\leq i\leq m-1.$ If $a_{1}\geq 0$ and $a_{m}\geq 0$, then $a_{i}\geq 0$ for $1\leq i\leq m$.
\end{lem}
\begin{proof}
We have $a_{i}-a_{i-1}\geq a_{i+1}-a_{i}$ for $2\leq i\leq m-1$, i.e.,
\begin{equation}\label{4.2.5}
a_{2}-a_{1}\geq a_{3}-a_{2}\geq \cdots \geq a_{m}-a_{m-1}.
\end{equation}
We claim that $a_{m-1}\geq 0$. In fact, assume not. Then by (\ref{4.2.5}), we obtain $$0\leq a_{1}<a_{2}< \cdots<a_{m-1}<0,$$ which gives a contradiction. Hence the claim holds. Now we have $a_1\geq 0$ and $a_{m-1}\geq 0,$ and the lemma follows by induction on $m.$
\end{proof}

\subsection{}In the present paper, we only focus on the case where $B=B(\mu)$ for $\mu\in\mathscr{P}_{N }.$ Then the corresponding multiplicity array $L(\mu)=(L^{(a)}_i)$  is given by $L^{(a)}_i=\delta_{a,1}m_i(\mu).$
Let $(\nu,J)$ be a $L(\mu)$-rigged configuration. We set $\nu^{(0)}=\mu$ and $\nu^{(n)}=\emptyset.$ By using (\ref{3.4.3}), it is easy to check that
\begin{align}
&\mathrm{wt}(\nu)=\sum_{a=1}^n(|\nu^{(a-1)}|-|\nu^{(a)}|)\overline{\epsilon}_a,\label{4.3.1}\\
&p_{i}^{(a)}=Q_{i}^{(a-1)}-2Q_{i}^{(a)}+Q_{i}^{(a+1)}, \text{ for } (a,i)\in\mathscr{H}.
\end{align}

\subsection{} A $L(\mu)$-rigged configuration $(\nu,J)$ is said to be valid if $x\leq p^{(a)}_i$ for any $a\in I_0,$ $(i,x)\in (\nu^{(a)},J^{(a)}).$
On the other hand, $(\nu,J)$ is said to be highest weight if $x\geq 0$ for all its labels. In particular, a valid highest weight $L(\mu)$-rigged configuration $(\nu,J)$ satisfies the condition
\begin{equation}\label{4.4.1}
0\leq x\leq p^{(a)}_i
\end{equation}
for any $a\in I_0,$ $(i,x)\in (\nu^{(a)},J^{(a)}).$

Let $(\nu,J)$ be a valid highest weight rigged configuration. For $a\in I_0,$ let $k_a$ be the largest length among the strings in $(\nu^{(a)},J^{(a)}).$ Then $Q_{k_a}(\nu^{(a)})=|\nu^{(a)}|.$ We have
\begin{eqnarray*}
0\leq p_{k_a}^{(a)} &=&Q_{k_a}^{(a-1)}-2Q_{k_a} ^{(a)}+Q_{k_a}^{(a+1)} \\
&=&Q_{k_a}^{(a-1)}-2|\nu^{(a)}|+Q_{k_a}^{(a+1)}  \\
&\leq&|\nu^{(a-1)}|-2|\nu^{(a)}|+|\nu^{(a+1)}|.
\end{eqnarray*}
Hence $(|\nu^{(0)}|-|\nu^{(1)}|,|\nu^{(1)}|-|\nu^{(2)}|,\ldots,|\nu^{(n-2)}|-|\nu^{(n-1)}|,|\nu^{(n-1)}|)\in \mathscr{P}_{N},$ i.e., $\mathrm{wt}(\nu,J)$ is a dominant weight.

For $\lambda\in\mathscr{P}_{N}$ such that $l(\lambda)\leq n,$ let $\mathrm{RC}(\mu,\lambda)$ denote the set of valid highest weight $L(\mu)$-rigged configurations such that $\mathrm{wt}(\nu,J)=\lambda.$

\begin{lem}\label{lem 4.4.1}
For $\lambda\in\mathscr{P}_{N }$ with $l(\lambda)\leq n,$ let $\mathrm{QM}(\mu,\lambda)$ be the set of $L(\mu)$-rigged configurations satisfying the condition
\begin{enumerate}
\item[(1)]
$|\nu^{(a)}|=\sum_{j>a}\lambda_j$ for $a\in I_0,$

\item[(2)]
 $p^{(a)}_i\geq 0$ for all $(a,i)\in \mathscr{H},$

\item[(3)]
$0\leq x\leq p^{(a)}_i$ for any  $a\in I_0,$ $(i,x)\in(\nu^{(a)},J^{(a)}).$
\end{enumerate}
Then $\mathrm{QM}(\mu,\lambda)=\mathrm{RC}(\mu,\lambda).$
\end{lem}
\begin{proof}
For a $L(\mu)$-rigged configuration $(\nu,J),$ $\mathrm{wt}(\nu,J)=\lambda$ if and only if $(\nu,J)$ satisfies the condition (1). Hence by (\ref{4.4.1}), the set $\mathrm{RC}(\mu,\lambda)$ is characterized by the properties (1), (2') and (3), where
\begin{enumerate}
\item[(2')]
 $p_i^{(a)}\geq 0$ for all $(a,i)\in\mathscr{H}$ such that $m_i^{(a)}>0.$
\end{enumerate}
Hence it is enough to show that $\mathrm{QM}(\lambda,\mu)\supseteq\mathrm{RC}(\mu,\lambda).$
Let $(\nu,J)\in \mathrm{RC}(\mu,\lambda).$  Note that $Q_{i}^{(a)}=|\nu^{(a)}|$ for all $a\in I_0\cup\{0\}$ if $i\gg 0.$  For such $i,$ we have
\begin{eqnarray}\label{4.4.2}
p_{i}^{(a)} &=&Q_{i}^{(a-1)}-2Q_{i} ^{(a)}+Q_{i}^{(a+1)} \\
&=&(\lambda _{a}+\lambda _{a+1}+\cdots )-2(\lambda _{a+1}+\lambda\nonumber
_{a+2}+\cdots )+(\lambda _{a+2}+\cdots ) \\ \nonumber
&=&\lambda _{a}-\lambda _{a+1}\geq 0\nonumber
\end{eqnarray}
if $1<a<n-1.$ This computation can be applied also for the case where $a=1$ or $a=n-1$ since $|\nu^{(0)}|=|\nu|=N=|\lambda|$ and $|\nu^{(n)}|=0=\Sigma_{j>n}\lambda_j.$
Now assume that $m_j^{(a)}=0.$ If there exists $k<j$ and $j<l$ such that $k=\nu^{(a)}_p,$ $l=\nu^{(a)}_q$ for some $p,q$, then $p_{k}^{(a)}\geq 0,p_{l}^{(a)}\geq 0$ by the condition (2'). Here we may assume that $m^{(a)}_u=0$ for $k<u<l.$ Then by Lemma \ref{Lem 4.1.1}, $p_k^{(a)},p_{k+1}^{(a)},\ldots,p_l^{(a)}$ satisfy the condition of Lemma \ref{Lem 4.1.2}. It follows, by Lemma \ref{Lem 4.1.2}, that $p_u^{(a)}\geq 0$ for $k\leq u\leq l.$ If there does not exist $j<l$ such that $l=\nu^{(a)}_q$ for some $q,$ then we can choose $l$ large enough so that $p_l^{(a)}\geq 0$ by (\ref{4.4.2}). Since we can find $k<j$ such that $k=\nu^{(a)}_p$ for some $p$ or otherwise $k=0,$ a similar argument as above can be applied.
\end{proof}

\subsection{} In Schilling [Sc1, 3.2], the Kashiwara operators $e_a,$ $f_a,$ ($a\in I_0$) are defined on the set of valid $L(\mu)$-rigged configurations.
Here, we just point out that for a valid $L(\mu)$-rigged configuration $(\nu,J)$ and $a\in I_0,$ $e_a(\nu,J)=0$ if and only if $x\geq 0$ for all the strings $(i,x)$ of $(\nu,J)^{(a)}.$ For other details, refer to [Sc1, 3.2].
Let $\mathrm{RC}(\mu)$ denote the set of valid $L(\mu)$-rigged configurations generated from all the valid highest weight $L(\mu)$-rigged configurations by the application of the operators $f_{a}$ and $e_{a},$ $a\in I_0.$   Define
\begin{align}
&\varepsilon _{a}(\nu ,J)=\max \{k \mid e_{a}^{k}(\nu ,J)\neq 0\}\label{F 4.8},\\
&\varphi_{a}(\nu ,J)=\max \{k\mid f_{a}^{k}(\nu ,J)\neq 0\}.\label{F 4.9}
\end{align}

The following result is essentially contained in Theorem 3.7 in [Sc1].
\begin{thm}[{[Sc1]}]\label{thm n4.3.1}
The maps $$\mathrm{wt}:\mathrm{RC}(\mu)\rightarrow \overline{P},\text{ } e_{i},f_{i}:\mathrm{RC}(\mu)\rightarrow \mathrm{RC}(\mu)\sqcup \{0\}, \text{ }\varepsilon _{i},\varphi_{i}:\mathrm{RC}(\mu)\rightarrow\mathbb{Z},\text{ } i\in I_0$$ define a $U_{q}(\mathfrak{g}_0)$-crystal structure on $\mathrm{RC}(\mu)$.
\end{thm}

\subsection{}In [KSS], a bijection
\begin{equation}
\Phi: B(\mu)\rightarrow \mathrm{RC}(\mu),
\end{equation} was constructed, called the rigged configuration bijection. In [Sa], Sakamoto proved that

\begin{thm}[{[Sa, Theorem 4.1]}]\label{n thm4.3.2}
$\Phi$ commutes with Kashiwara operators. In particular, $\Phi$ gives a $U_{q}(\mathfrak{g}_0)$-crystal isomorphism.
\end{thm}
\begin{rem}
\emph{Here we shall explain the construction of the map $\Psi.$ Recall the bijection $\Psi:\mathrm{W}(\mu)\rightarrow B(\mu)$ defined in 3.4. We define a map $\psi:\mathrm{W}(\mu)\rightarrow B(\mu)$ by $\psi=\Phi\circ\Psi.$ Then by Theorem \ref{cor 3.3.4} and Theorem \ref{n thm4.3.2}, $\psi$ gives a $U_{q}(\mathfrak{g}_0)$-crystal isomorphism. By the bijection $\Psi,$ the construction of $\Phi$ is easily translated to the construction of $\psi,$ step by step. So here we only explain the construction of $\psi.$ The verification of $\psi=\Phi\circ \Psi$ is easy.}

\emph{For a partition $\lambda$ and an integer $m$ such that $1\leq m\leq |\lambda|,$ we define $\lambda-[m]$ inductively as follows; put $\lambda-[1]=(\lambda_1,\ldots,\lambda_{k-1},\lambda_k-1)$ if $\lambda=(\lambda_1,\ldots,\lambda_k)$ with $\lambda_k>0$ and define $\lambda-[m]=(\lambda-[m-1])-[1]$ for $m\geq 2.$ Note that $\lambda-[m]$ is the partition whose Young diagram is
obtained by omitting $m$ boxes from $\lambda $ from bottom to top, right
to left.}

\emph{For a rigged configuration $(\nu,J)$, it is helpful to formally regard the rigged partition $\nu^{(k)}$ as the complement of the ($l(\nu^{(k)})+1$)-th row of length zero, which corresponds to the singular string of length zero. We denote such string by $S^{(k)}_{0}$.}

\emph{Let $\BBw=(w_{n},w_{n-1},\cdots ,w_{1})\in \mathrm{W}(\mu).$ We say $\BBw$
is of rank $r$, denoted by $\mathrm{rank}(\BBw)=r$, if the maximal letter in the word $w_{n}\ast
w_{n-1}\ast \cdots \ast w_{1}$ first appears in the row $w_{r+1}$ in the
sequence $w_{1},\cdots ,w_{n}$. Assume that $w_{r+1}=a_p\cdots a_2a_1,$ and put $w'_{r+1}=a_p\cdots a_2.$ Set $\widetilde{\BBw}=(w_n,\ldots,w'_{r+1},\ldots,w_1).$
Then $\BBw$ determines a sequence $\BBw_0,\ldots,\BBw_N=\BBw,$ where $\BBw_{j-1}=\widetilde{\BBw}_j,$ for $1\leq j\leq N.$
Note that $\widetilde{\BBw}_i\in\mathrm{W}(\mu-[N-i]).$}

\emph{We construct $(\nu,J)=\psi(\BBw)$ inductively as follows.
Begin with $(\nu,J)_0=\emptyset\in\mathrm{RC}(\nu-[N]).$ Suppose that $(\nu,J)_{i-1}\in\mathrm{RC}(\mu-[N-i+1])$ has been constructed for $i\geq 1$. Assume that $\BBw_i$ is of rank $r.$}

\begin{enumerate}

\item[(1)]
 \emph{If $r=0,$ set $(\nu,J)_i=(\nu,J)_{i-1}.$}

\item[(2)] \emph{If $r>0,$ we choose a singular string $S^{(r)}=(i_r,x_r)$  of $(\nu^{(r)},J^{(r)})_{i-1}$ with largest length. If there exists no such a string, let $S^{(r)}=S^{(r)}_0$ be the string of length 0. Suppose that the string $S^{(k)}=(i_k,x_k)$ has been already chosen.  We choose $S^{(k-1)}=(i_{k-1},x_{k-1})$ to be a string with largest length among the singular strings $(i,x)$ of $(\nu^{(k-1)},J^{(k-1)})_{i-1}$ such that $i\leq i_k.$ If there exists no such string, let $S^{(k-1)}=S^{(k-1)}_0$ be the string of length 0.
$(\nu,J)_i$ is obtained from $(\nu,J)_{i-1}$ by adding the length of the chosen strings by one, setting them singular and keeping the remaining strings.}

\end{enumerate}
\emph{The process ends at $(\nu,J)_N.$ Then we put $(\nu,J)=(\nu,J)_N.$}
\end{rem}

\begin{rem}
\emph{In (2) of Remark 4.6.2, assume that $\mu-[|\mu|-i]=(\mu_{1},\ldots,\mu_{p-1},\mu'_{p}).$} \emph{Set $i_0=\mu'_{p}-1.$ Then it is easy to see $i_0\leq i_1.$}
\end{rem}

\begin{rem}
\emph{Let $\lambda\in\mathscr{P}_N.$ For $T\in\mathrm{Tab}(\lambda,\mu),$ let $\BBw_T\in \mathrm{W}(\mu)$ as in 3.6. Then $\mathrm{Tab}(\lambda,\mu)$ can be } \emph{identified with a subset of $\mathrm{W}(\mu)$ by the embedding $\mathrm{Tab}(\lambda,\mu)\hookrightarrow \mathrm{W}(\mu),$ $T\mapsto \BBw_T.$}
\emph{In [KR1], a bijection $\Pi_\lambda:\mathrm{Tab}(\lambda,\mu)\isom\mathrm{RC}(\mu,\lambda)$ was constructed.  On the other hand, the bijection} \emph{$B(\mu)\rightarrow\mathrm{RC}(\mu)$ induces a bijection $$P(B(\mu),\lambda)\isom\mathrm{RC}(\mu,\lambda),$$ hence combined with the bijection} \emph{$\mathrm{Tab}(\lambda,\mu)\isom\mathcal{P}(B(\mu),\lambda)$ given in (\ref{3.6.2}), we obtain a bijection } \emph{$\mathrm{Tab}(\lambda,\mu)\isom\mathrm{RC}(\mu,\lambda).$ This bijection is nothing but $\Pi_\lambda.$ Thus our map } \emph{$\psi:\mathrm{W}(\mu)\rightarrow\mathrm{RC}(\mu)$ can be viewed as an extension of $\Pi_\lambda$ to the whole set $\mathrm{W}(\mu).$}
\end{rem}

\subsection{}In what follows, we fix $\boldsymbol{\lambda}=(\lambda',\lambda'')\in \mathscr{P}_{N,2}$ with $\lambda'=(\lambda'_{1},\lambda'_{2},\ldots,\lambda'_s),$ $\lambda''=(\lambda''_{1},\lambda''_{2},\ldots,\lambda''_t)$ such that $s+t\leq n.$ Set
\begin{equation}
\mathrm{RC}(\mu,\boldsymbol{\lambda})=\{\psi(\BBw_T)\mid T\in \mathrm{Tab}(\boldsymbol{\lambda} ,\mu )\}.
\end{equation}
 By Proposition \ref{cor 3.4.2}, we have
\begin{equation}\label{F 4.15}
\mathrm{RC}(\mu,\boldsymbol{\lambda})\isom P(B(\mu),\boldsymbol{\lambda})
\end{equation} since $\psi$ is a crystal isomorphism.

Let $\lambda=(\lambda_1,\ldots,\lambda_n)$ and $\rho=(\rho_1,\ldots,\rho_n)$ be partitions such that $\rho\subset\lambda$ and that $|\lambda|-|\rho|=N.$ We further assume that $\rho$ is a rectangle type. We define $\mathrm{QM}(\mu,\lambda- \rho)$ as the set of $L(\mu)$-rigged configurations satisfying the conditions

\begin{enumerate}
\item[(1)]
 $|\nu ^{(a)}|=\sum\nolimits_{j\geq a+1}(\lambda _{j}-\rho _{j})$ for $a\in I_0,$
\item[(2)]
$\widetilde{p}_{i}^{(a)}\geq 0$ for all $(a,i)\in\mathscr{H},$ where
\begin{equation}
\widetilde{p}_{i}^{(a)}=p_i^{(a)}+\min(i,\rho_a-\rho_{a+1}).
\end{equation}

\item[(3)]
 $0\leq x\leq \widetilde{p}_{i}^{(a)}$ for any $a\in I_0$ and any string $(i,x)$ of $(\nu^{(a)},J^{(a)}).$

\end{enumerate}
In Corollary 1 of Theorem 2.1 in [KR1], Kirillov and Reshetikhin proved the following result.
\begin{thm}[{[KR, Corollary 1]}]\label{lem 4.5.1}
 $|\mathrm{Tab}(\lambda-\rho,\mu)|=|\mathrm{QM}(\mu,\lambda-\rho)|.$
\end{thm}
\subsection{} Let $\boldsymbol{\lambda}=(\lambda',\lambda'')\in \mathscr{P}_{N,2}$ be as in 4.7. We define $\mathrm{QM}(\mu,\boldsymbol{\lambda})$ as the set of  $L(\mu)$-rigged configurations $(\nu,J)$ satisfying the condition
\begin{enumerate}
\item[(1)]
 $\mathrm{wt}(\nu,J)=\boldsymbol{\lambda};$

\item[(2)]
 $0\leq x\leq p^{(a)}_{i}+\delta_{a,s}i$ for any $a\in I_0$ and any string $(i,x)$ of $(\nu^{(a)},J^{(a)}).$
\end{enumerate}

We have the following lemma

\begin{lem}\label{lem 4.5.2}
 Let $\boldsymbol{\lambda}=(\lambda',\lambda'')\in \mathscr{P}_{N,2},$ with $\lambda'=(\lambda'_1,\ldots,\lambda'_s),$ $\lambda''=(\lambda''_1,\ldots,\lambda''_t),$ such that $s+t\leq n.$ Then
\begin{equation}
|\mathrm{Tab}(\boldsymbol{\lambda},\mu)|=|\mathrm{QM}(\mu,\boldsymbol{\lambda})|.
\end{equation}
\end{lem}

\begin{proof}
We follow the notation in 2.2. Let us choose an integer $m\geq N,$ and consider the partitions $\xi=\widetilde{\xi}_{\Bla,m}$ and $\rho=(m^s).$ Then we can identify the set $\mathrm{Tab}(\Bla,\mu)$ with the set $\mathrm{Tab}(\xi-\rho,\mu).$ By Theorem \ref{lem 4.5.2}, $\mathrm{Tab}(\boldsymbol{\lambda},\mu)$ and $\mathrm{QM}(\mu,\xi-\rho)$ have the same cardinality. Hence, it is enough to show that $\mathrm{QM}(\mu,\boldsymbol{\lambda})=\mathrm{QM}(\mu,\xi-\rho).$
Let $(\nu,J)$ be a $L(\mu)$-rigged configuration. Note, by (\ref{4.3.1}), that $\mathrm{wt}(\nu,J)=\boldsymbol{\lambda}$ if and only if $|\nu ^{(a)}|=\sum\nolimits_{j\geq a+1}(\xi _{j}-\rho _{j})$. Hence we have $\mathrm{QM}(\mu,\xi-\rho)\subset \mathrm{QM}(\mu,\boldsymbol{\lambda}).$
In order to show the reverse inclusion $\mathrm{QM}(\mu,\boldsymbol{\lambda})\subset \mathrm{QM}(\mu,\xi-\rho),$ we shall show that if $(\nu,J)\in \mathrm{QM}(\mu,\boldsymbol{\lambda}),$ then $\widetilde{p}^{(a)}_{i}\geq 0$ for all $(a,i)\in\mathscr{H}$. By our choice of $\rho,$ we have $\widetilde{p}^{(a)}_{i}=p^{(a)}_{i}+\delta_{s,a}\min (i,m).$
Let $(b_{1},b_{2},\ldots, b_{s+t})=(\lambda'_{1},\lambda'_{2},\ldots, \lambda'_{s},\lambda''_{1},\ldots,\lambda''_{t}).$

First assume that $i\geq m.$ Since $|\nu^{(a)}|\leq N$ for $a\in I_0\cup \{0\},$ and $m\geq N,$ we see that $Q_i^{(a)}=|\nu^{(a)}|$ and $\min(i,m)=m.$ It follows, for $a\in I_0,$ that
\begin{eqnarray*}
\widetilde{p}^{(a)}_{i} &=&p_{i}^{(a)}+\delta _{a,s}\min (i,m) \\
&=&Q_{i}^{(a-1)}-2Q_{i}^{(a)}+Q_{i}^{(a+1)}+\delta _{a,s}m \\
&=&|\nu ^{(a-1)}|-2|\nu ^{(a)}|+|\nu ^{(a)}|+\delta _{a,s}m \\
&=&\begin{cases}
                b_a-b_{a+1}+m   &\quad\text{ if $a=s$}, \\
                b_a-b_{a+1}       &\quad\text{ if $a\neq s$. }
        \end{cases}
\end{eqnarray*}
If $a\neq s,$ $b_a-b_{a+1}$ coincides with $\lambda'_j-\lambda'_{j+1}$ or $\lambda''_j-\lambda''_{j+1}$ for some $j,$ hence $\widetilde{p}^{(a)}_{i}\geq 0.$ If $a=s,$ $b_a-b_{a+1}+m=\lambda'_s-\lambda''_1+m\geq 0$ since $m\geq N.$ Thus $\widetilde{p}^{(a)}_{i}\geq 0.$

Next assume that $i<m.$ If $m_i^{(a)}\neq 0,$ then $\widetilde{p}_{i}^{(a)}\geq 0$ by (2) in 4.8. Hence we may assume that $m_i^{(a)}=0.$ By applying (\ref{4.2.4}) to $p_i^{(a)}$, we have
\begin{eqnarray*}
2\widetilde{p}_{i}^{(a)} &=&2p_{i}^{(a)}+2\delta _{a,s}\min (i,m) \\
&=&2p_{i}^{(a)}+2\delta _{a,s}i \\
&\geq &p_{i-1}^{(a)}+p_{i+1}^{(a)}+2\delta _{a,s}i \\
&=&\left(p_{i-1}^{(a)}+\delta _{a,s}(i-1)\right)+\left(p_{i+1}^{(a)}+\delta
_{a,s}(i+1)\right) \\
&=&\widetilde{p}_{i-1}^{(a)}+\widetilde{p}_{i+1}^{(a)}.
\end{eqnarray*}
Since $\widetilde{p}_{i}^{(a)}\geq 0$ if $m_{i}^{(a)}\neq0$, by a similar argument as in the proof of Lemma \ref{lem 4.4.1}, we conclude that  $\widetilde{p}_{i}^{(a)}\geq 0.$  Thus $\widetilde{p}_{i}^{(a)}\geq 0$ for any $(a,i)\in\mathscr{H},$ and the lemma follows.
\end{proof}

\subsection{}Let $\BBw\in\mathrm{W}(\mu).$ Recall the construction of $(\nu,J)=\psi(\BBw)$ as given in Remark 4.6.2, in particular the definition of $(\nu,J)_i$ given there.  Let $s\in I_0$ be as in 4.7. For $j=0,\ldots,N$ we consider the following condition $P(\BBw,s,j).$
\begin{equation}\label{4.9.1}
P(\BBw,s,j)\text{: } 0\leq x+i\leq p^{(s)}_{i}+i \text{ holds for all the strings } (i,x) \text{ of } (\nu^{(s)},J^{(s)})_j.
\end{equation}

\begin{lem}\label{Lem 4.5.3}
Assume that $0\leq j< N.$ If $P(\BBw,s,j)$ holds, then $P(\BBw,s,j+1)$ follows.
\end{lem}

\begin{proof}
Write $(\nu,J)_k$ as $(\nu,J)_k=((\nu^{(1)}(k),J^{(1)}(k)),\ldots,(\nu^{(n)}(k),J^{(n)}(k))).$ We fix $j$ and put
$\overline{Q}^{(a)}_{i}=Q_{i}(\nu^{(a)}(j))$ and $Q^{(a)}_{i}=Q_i(\nu^{(a)}(j+1)),$ for simplicity. For $(a,i)\in\mathscr{H},$ let $\overline{p}^{(a)}_i$(resp. $p^{(a)}_i$) be the vacancy numbers with respect to $(\nu,J)_j$ (resp. $(\nu,J)_{j+1}$).
Put $\mathrm{rank}(\BBw_{j+1})=r.$ For $k=1,\ldots,r,$ let $S^{(k)}=(i_k,x_k),$ $1\leq k\leq r,$ be the string of $(\nu,J)_j$ chosen in the construction of $(\nu,J)_{j+1}$ in Remark 4.6.2. For $k>r,$ set $i_k=\infty.$ By expressing $\mu-[N-i]$ as $(\mu_1,\ldots,\mu_{p-1},\mu'_p),$ set $i_0=\mu'_p-1.$ Then $i_0\leq i_1\leq\cdots\leq i_n.$  For an integer $x,$ let $\delta(x)=1$ (resp. $0$) if $x>0$ (resp. $x\leq 0$).
Then $Q^{(a)}_{i}=\overline{Q}^{(a)}_{i}+\delta(i-i_{a})$ for each $a\in I_0.$ Hence
\begin{eqnarray*}
p_{i}^{(s)} &=&Q_{i}^{(s-1)}-2Q_{i}^{(s)}+Q_{i}^{(s+1)} \\
&=&(\overline{Q}_{i}^{(s-1)}+\delta (i-i_{s-1}))-2(\overline{Q}%
_{i}^{(s)}+\delta (i-i_{s}))+(\overline{Q}_{i}^{(s+1)}+\delta (i-i_{s+1})) \\
&=&(\overline{Q}_{i}^{(s-1)}-2\overline{Q}_{i}^{(s)}+\overline{Q}%
_{i}^{(s+1)})+(\delta (i-i_{s-1})-2\delta (i-i_{s})+\delta (i-i_{s+1})) \\
&=&\overline{p}_{i}^{(s)}+(\delta (i-i_{s-1})-2\delta (i-i_{s})+\delta
(i-i_{s+1})).
\end{eqnarray*}
Thus we obtain the following formulas.
\begin{equation}\label{4.9.2}
p^{(s)}_{i}=\overline{p}^{(s)}_{i}+\left\{
\begin{array}{ll}
0 &\mbox{if $i\leq i_{s-1}$,}\\
1 &\mbox{if $i_{s-1}<i\leq i_{s}$,}\\
-1&\mbox{if $i_{s}<i\leq i_{s+1}$,}\\
0 &\mbox{if $i_{s+1}\leq i.$}
\end{array}
\right.
\end{equation}
Since $0\leq x\leq p_i^{(s)},$ in order to show $P(\BBw,s,j+1),$ it is enough to see, by (\ref{4.9.2}), that $\overline{p}^{(s)}_{i}+i>0$ for $(i,x)\in (\nu,J)^{(s)}_i$ with $i_{s}<i\leq i_{s+1}.$ We prove this separately according to the following 4 cases. In Case 1-3, it is assumed that $i_{s+1}\neq\infty,$ while in Case 4, assumed that $i_{s+1}=\infty.$

\par\smallskip\noindent\textbf{Case 1. $i_{s}+1<i.$}

We have $m_{i}(\nu^{(s)}(j))>0$ in this case. The strings $(i,x)$ of $(\nu^{(s)},J^{(s)})_j$ must be nonsingular, i.e. $\overline{p}^{(s)}_{i}+i>x+i\geq 0.$

\par\smallskip\noindent\textbf{Case 2.} $i=i_{s}+1$, $m_{i}(\nu^{(s)}(j))=0$ and $m_{k}(\nu^{(s)}(j))\neq0$ for some $i\leq k\leq i_{s+1}.$

In this case, the strings $(k,y)$ of $\nu^{(s)}(j)$ must be nonsingular, i.e. $\overline{p}^{(s)}_{k}+k>y+k\geq 0.$ Use a similar argument as in the proof of Lemma \ref{lem 4.4.1}, we have $\overline{p}^{(s)}_{i}+i>0.$

\par\smallskip\noindent\textbf{Case 3.} $i=i_{s}+1$ and $m_{k}(\nu^{(s)}(j))=0$, for $i\leq k\leq i_{s+1}.$

Let $c=i_{s+1}-i.$ By applying (\ref{4.2.4}) in Lemma \ref{Lem 4.1.1} successively, we have
\begin{eqnarray}\label{4.9.3}
2^{c+1}\overline{p}_{i}^{(s)} &=&2^{c}(2\overline{p}%
_{i}^{(s)}) \\ \nonumber
&\geq& 2^{c}\overline{p}_{i-1}^{(s)}+2^{c-1}(2%
\overline{p}_{i+1}^{(s)}) \\ \nonumber
&\cdots& \\ \nonumber
&\geq &2^{c}\overline{p}_{i-1}^{(s)}+2^{c-1}\overline{p}%
_{i}^{(s)}+\cdots +2\overline{p}_{i+c-2}^{(s)}+2\overline{%
p}_{i+c}^{(s)}. \nonumber
\end{eqnarray}
On the other hand, by (\ref{4.2.3}) of Lemma \ref{Lem 4.1.1} and by our assumption $m_{i+c}(\nu^{(s)}(j))=m_{i_{s+1}}(\nu^{(s)}_j)=0,$ we have
\begin{equation}\label{4.9.4}
2\overline{p}_{i+c}^{(s)}\geq \overline{p}_{i+c-1}^{(s)}+\overline{p}_{i+c+1}^{(s)}+m_{i+c}(\nu^{(s-1)}(j))+m_{i+c}(\nu^{(s+1)}(j)).
\end{equation}
Put $F_i=\overline{p}_{i}^{(s)}+i.$ Since $m_{i_s+1}(\nu^{(s+1)}(j))>0,$ by substituting (\ref{4.9.4}) into (\ref{4.9.3}), we have
\begin{equation*}
2^{c+1}F_i\geq 2^cF_{i-1}+\cdots+2F_{i+c-2}+F_{i+c-1}+F_{i+c+1}+m_{i+c}(\nu^{(s+1)}(j))>0.
\end{equation*}
This shows that $\overline{p}_{i}^{(s)}+i>0.$

\par\smallskip\noindent\textbf{Case 4.} $i_{s+1}=\infty.$

In this case, one can find some $u>i$ such that $\overline{p}_{u}^{(s)}+u>0$ and $m_{k}(\nu^{s}(j))=0$ for $i_{s}\leq k< u$. Then by applying the method used in the proof of Lemma \ref{lem 4.4.1}, we obtain $\overline{p}_{i}^{(s)}+i>0$.
\end{proof}

Let $\Bla=(\lambda',\lambda'')\in\mathscr{P}_{N,2}$ as in 4.7.  For $(\nu,J)\in\mathrm{RC}(\mu,\Bla),$ let $(\nu,J_{+})$ be the $\mu$-rigged configuration obtained by replacing the strings $(i,x)$ of $(\nu^{(s)},J^{(s)})$ by $(i,i+x).$

\begin{prop}\label{prop 4.5.4}
$(\nu,J_{+})\in \mathrm{QM}(\mu,\boldsymbol{\lambda})$ for $(\nu,J)\in \mathrm{RC}(\mu,\boldsymbol{\lambda}).$ In particular, the map $(\nu,J) \mapsto (\nu,J_{+})$ gives a bijection $\mathrm{RC}(\mu,\boldsymbol{\lambda})\isom\mathrm{QM}(\mu,\boldsymbol{\lambda}).$
\end{prop}
\begin{proof}
Assume that $(\nu,J)=\psi(\BBw_T),$ with $\BBw_T=(w_n,\ldots,w_1)$ for $T\in\mathrm{Tab}(\Bla,\mu).$ It is enough to show $0\leq x+i\leq p^{(a)}_i+i$ for all the strings $(i,x)$ of $(\nu^{(s)},J^{(s)}).$

Recall the construction of $(\nu,J)$ given in Remark 4.6.2. Note that $\BBw_1=(1,-,\ldots,-)$ or $(\overbrace{-,\ldots,-}^{s},
1,-\ldots,-).$ It is easy to check that $P(\BBw_T,s,0)$ in (\ref{4.9.1}) holds in each case. Hence, $P(\BBw_T,s,N)$ holds by Lemma \ref{Lem 4.5.3}. Namely, $0\leq x+i\leq p^{(a)}_i+i$ for all the strings $(i,x)$ of $(\nu,J)^{(s)}.$
Notice that the map $(\nu,J)\mapsto (\nu,J_+)$ is injective, hence it gives a bijection since both sides have the same cardinality by Lemma \ref{lem 4.5.2} and (\ref{F 4.15}).
\end{proof}

We have the following corollary immediately.
\begin{cor}
 $\mathrm{RC}(\mu,\boldsymbol{\lambda})$ consists  of the $L(\mu)$-rigged configurations $(\nu,J)$ satisfying the condition
\begin{enumerate}
\item[(1)]
 $\mathrm{wt}(\nu,J)=\boldsymbol{\lambda};$

\item[(2)]
 $0\leq x+\delta_{a,s}i\leq p^{(a)}_{i}+\delta_{a,s}i$ for any $a\in I_0$ and any string $(i,x)$ of $(\nu^{(a)},J^{(a)}).$
\end{enumerate}
\end{cor}

\subsection{}A $L(\mu)$-configuration $\nu$ is said to be admissible if $p_i^{(a)}\geq 0$ for all $(i,a)\in\mathscr{H}.$ For $\lambda\in\mathscr{P}_{N}$ with $l(\lambda)\leq n,$ let $\mathrm{C}(\mu,\lambda)$ be the set of admissible $L(\mu)$-configurations of weight $\lambda.$ Consider the $t$-binomial coefficient
$$\left[ { p+m \atop m} \right]
_t=\frac{(t)_{p+m}}{(t)_p(t)_m},$$ where $(t)_m=(1-t)(1-t^2)\cdots(1-t^m).$ For $\lambda\in\mathscr{P}_{N}$ with $l(\lambda)\leq n,$
the fermionic formula with respect to $\lambda$ is given by
\begin{equation}
M(\mu,\lambda;t) =\sum_{\{m\}}t^{\mathrm{cc}(\{m\})}\prod_{(a,i)\in\mathscr{H}}\left[ { p^{(a)}_i+m_i^{(a)} \atop m_i^{(a)} } \right]_t,\label{4.10.2}
\end{equation}
where $\mathrm{cc}(m)$ and $p_i^{(a)}$ are given by
\begin{align}
&\mathrm{cc}(\{m\}) =\frac{1}{2}\sum\limits_{a,b\in I_0,i,j\geq 1}(\alpha_a,\alpha_b)\min(i,j)m_{i}^{(a)}m_{j}^{(b)},\nonumber\\
&p_{i}^{(a)}=\delta_{a,1}\sum_{j=1}^r\min(i,\mu_j)-\sum_{(b,j)\in\mathscr{H}}(\alpha_a,\alpha_b)\min(i,j)m_j^{(b)}.\nonumber
\end{align}Here the sum $\sum_{\{m\}}$ is taken over $\{m^{(b)}_j\in\mathbb{Z}_{\geq 0}\mid (b,j)\in\mathscr{H}\}$ such that $p^{(a)}_i\geq 0$ for any $(a,i)\in\mathscr{H}$ and $N\overline{\Lambda}_1-\sum_{(a,i)\in\mathscr{H}}im_i^{(a)}\alpha_a=\lambda.$ Note that such $\{m\}$ can be naturally identified with an admissible  $L(\mu)$-configuration of weight $\lambda,$ and vice versa. For a $L(\mu)$-configuration, define the cocharge $\mathrm{cc}(\nu)$ of $\nu$ as
\begin{equation}
\mathrm{cc}(\nu)=\frac{1}{2}\sum\limits_{a,b\in I_0,i,j\geq 1}(\alpha_a,\alpha_b)\min(i,j)m_{i}^{(a)}m_{j}^{(b)}.
\end{equation} Then we can rewrite (\ref{4.10.2}) as
\begin{equation}\label{nf4.10.2}
M(\mu,\lambda;t)=\sum_{\nu\in\mathrm{C}(\mu,\lambda)}t^{\mathrm{cc}(\nu)}\prod_{(a,i)\in\mathscr{H}}\left[ { p^{(a)}_i+m_i^{(a)} \atop m_i^{(a)} } \right]_{t}.
\end{equation}

In [KR1], Kirillov and Reshetikhin gave an expression of the Kostka polynomials in terms of fermionic formula.
\begin{thm}[{[KR1]}]
Let $\lambda,\mu\in\mathscr{P}_{N },$ with $l(\lambda)\leq n$. Then
\begin{equation}\label{F 4.11}
M(\mu,\lambda;t)=t^{n(\mu)}K_{\lambda,\mu}(t^{-1}),
\end{equation}
where $n(\mu)=\sum_{j=1}^k(j-1)\mu_j$ for $\mu=(\mu_1,\ldots,\mu_k).$
\end{thm}

Following [Sc1, (3.5)], we define a cocharge $\mathrm{cc}(\nu,J)$ for $(\nu,J)\in\mathrm{RC}(\mu)$ by
\begin{equation}
\mathrm{cc}(\nu,J)=\frac{1}{2}\sum\limits_{a,b\in I_0,i,j\geq 1}(\alpha_a,\alpha_b)\min(i,j)m_{i}^{(a)}m_{j}^{(b)}+|J|
\end{equation}
 where $|J|$ denote the sum of all labels of the strings of $(\nu,J)$. Then (\ref{nf4.10.2}) can be reformulated as the generating function over the highest weight rigged configurations of weight $\lambda$
\begin{equation}
M(\mu, \lambda;t)=\sum\limits_{(\nu ,J)\in \mathrm{RC}(\mu ,\lambda )}t^{%
\mathrm{cc}(\nu ,J)}.
\end{equation}
The following result was proved by Schilling in [Sc1, Theorem 3.9].
\begin{thm}[{[Sc1]}]\label{thm 4.3.3}
The cocharge function is constant on each $\mathrm{U}_{q}(\mathfrak{g}_{0})$-connected component of $\mathrm{RC}(\mu)$.
\end{thm}
\subsection{} We now extend the fermionic formulas for Kostka polynomials to the case of double Kostka polynomials. First we define $M(\Bla,\mu;t)$ as an analogue of $M(\lambda,\mu;t)$ as follows. For $\Bla=(\lambda',\lambda'')\in\mathscr{P}_{N,2}$ as in 4.7,  set
\begin{equation}\label{4.11.1}
M(\mu,\Bla;t)=\sum_{(\nu,J)\in\mathrm{RC}(\mu,\Bla)}t^{\mathrm{cc}(\nu,J)}.
\end{equation}
We define $\mathrm{C}(\mu,\Bla)$ to be the set of $L(\mu)$-configurations $\nu$ satisfying the condition

\begin{enumerate}

\item[(1)]
 $\mathrm{wt}(\nu)=\Bla;$

\item[(2)]
 $0\leq p_i^{(a)}+\delta_{a,s}i$ for any $(a,i)\in\mathscr{H}$ such that $m_i^{(a)}\neq0.$

\end{enumerate}
By a similar argument as in the proof of Lemma \ref{lem 4.4.1}, it is easy to see that (2) can be replaced by
\begin{enumerate}
\item[(2')]
$0\leq p_i^{(a)}+\delta_{a,s}i$ for any $(a,i)\in\mathscr{H}.$
\end{enumerate}
By Proposition \ref{prop 4.5.4}, the map $\mathrm{RC}(\mu,\Bla)\rightarrow \mathrm{C}(\mu,\Bla),$ $(\nu,J)\mapsto \nu$ is surjective. For $\nu\in\mathrm{C}(\mu,\Bla),$ set $$\mathrm{cc}(\nu)=\frac{1}{2}\sum\limits_{a,b\in I_0,i,j\geq 1}(\alpha_a,\alpha_b)\min
(i,j)m_{i}^{(a)}m_{j}^{(b)}-|\nu^{(s)}|.$$
Then for $(\nu,J)\in\mathrm{RC}(\mu,\Bla),$ we have $$\mathrm{cc}(\nu,J)=\mathrm{cc}(\nu)+|J|+|\nu^{(s)}|.$$

\begin{lem}\label{lem 4.11.1}
For $\Bla=(\lambda',\lambda'')\in\mathscr{P}_{N,2}$ as in 4.7, we have
\begin{equation}\label{4.11.2}
M(\mu,\Bla;t)=\sum_{\nu\in\mathrm{C}(\mu,\Bla)}t^{\mathrm{cc}(\nu)}\prod_{(a,i)\in\mathscr{H}}\left[ { p^{(a)}_i+\delta_{a,s}+m_i^{(a)} \atop m_i^{(a)} } \right]_t.
\end{equation}
\end{lem}
\begin{proof}
We have
\begin{eqnarray}\label{4.12.3}
M(\mu ,\boldsymbol{\lambda} ;t)&=&\sum\limits_{\nu\in\mathrm{C}(\mu,\boldsymbol{\lambda})}t^{\mathrm{cc}(\nu,J)+|\nu^{(s)}|}\sum\limits_{(\nu,J)\in\mathrm{RC}(\mu,\Bla)}t^{|J|}\\
                                &= &\sum\limits_{\nu\in\mathrm{C}(\mu,\boldsymbol{\lambda})}t^{\mathrm{cc}(\nu,J)+|\nu^{(s)}|}\prod\limits_{(a,i)\in\mathscr{H}}\sum\limits_{J^{(a,i)}}t^{|J^{(a,i)}|}\nonumber
\end{eqnarray}
where $J^{(a,i)}$ runs over all the riggings of the rows of length $i$ in $\nu^{(a)}.$ By using the bijection in Proposition \ref{prop 4.5.4}, we see that the sum $\sum t^{|J^{(a,i)}|}$ coincides with $t^{-|\nu^{(s)}|}\sum t^{|J_+^{(a,i)}|},$ where the set of such $J_+^{(a,i)}$ can be identified with the set of partitions with at most $m_i^{(a)}$ parts each not exceeding $p_i^{(a)}+\delta_{a,s}i$ and $|J_+^{(a,i)}|$ coincides with the size of the corresponding partition. Hence by a well-known formula for the generating function of such partitions, we have
\begin{equation}\label{4.12.4}
\sum_{J^{(a,i)}}t^{|J^{(a,i)}|}=t^{-|\nu^{(s)}|}\left[ { p^{(a)}_i+\delta_{a,s}+m_i^{(a)} \atop m_i^{(a)} } \right]_t.
\end{equation}
Substituting (\ref{4.12.4}) into (\ref{4.12.3}), we obtain (\ref{4.11.2}).
\end{proof}
Note that $\nu\in\mathrm{\mu,\Bla}$ is completely determined by the corresponding $m_i^{(a)},(a,i)\in\mathscr{H}.$ By (\ref{4.11.2}), we can redefine $M(\mu,\Bla;t)$ as follows.
\begin{Def} \emph{Let $\mu\in\mathscr{P}_{N}$ and $\Bla\in\mathscr{P}_{N,2}$ as in 4.7. The fermionic formula $M(\mu,\Bla;t)$ is defined as}
\begin{equation}
M(\mu,\Bla;t) =\sum_{\{m\}}t^{\mathrm{cc}(\{m\})}\prod_{(a,i)\in\mathscr{H}}\left[ { p^{(a)}_i+\delta_{a,s}+m_i^{(a)} \atop m_i^{(a)} } \right]_t,
\end{equation}
\emph{where $\mathrm{cc}(m)$ and $p_i^{(a)}$ are defined by}
\begin{align*}
&\mathrm{cc}(\{m\}) =\frac{1}{2}\sum\limits_{a,b\in I_0,i,j\geq 1}(\alpha_a,\alpha_b)\min(i,j)m_{i}^{(a)}m_{j}^{(b)}-\sum_{i\geq 1}im^{(s)}_i,\nonumber\\
&p_{i}^{(a)}=\delta_{a,1}\sum_{j=1}^r\min(i,\mu_j)-\sum_{(b,j)\in\mathscr{H}}(\alpha_a,\alpha_b)\min(i,j)m_j^{(b)}.\nonumber
\end{align*}
\emph{Here, the sum $\sum_{\{m\}}$ is taken over $\{m^{(b)}_j\in\mathbb{Z}_{\geq 0}\mid (b,j)\in\mathscr{H}\}$ such that $p^{(a)}_i+\delta_{a,s}i\geq 0$ for any $(a,i)\in\mathscr{H}$ and $N\overline{\Lambda}_1-\sum_{(a,i)\in\mathscr{H}}im_i^{(a)}\alpha_a=\Bla.$}
\end{Def}
The following theorem is our main result, which gives an analogue of (\ref{F 4.11}) for double Kostka polynomials.
\begin{thm}\label{thm 4.6.3} Let $\Bla=(\lambda',\lambda''),\Bmu=(-,\mu'')\in\mathscr{P}_{N,2}$ with $\lambda'=(\lambda'_1,\ldots,\lambda'_s),$ $\lambda''=(\lambda''_1,\ldots,\lambda''_t).$ Assume that $s+t\leq n.$ Then
\begin{equation}
M(\mu'',\Bla;t^2)=t^{2n(\mu'')+|\lambda'|}K_{\Bla,\Bmu}(t^{-1}).
\end{equation}
\end{thm}
\begin{proof}
For a tableau $T$ of weight $\mu''$ with at most $n$ rows, set $\mathrm{cc}(T)=\mathrm{cc}(\Psi(\BBw_T)).$ Since $\Psi$ is an $U_{q}(\mathfrak{g}_0)$-crystal isomorphism by Proposition \ref{cor 3.3.5} and Theorem \ref{thm 4.3.3}, we have $\mathrm{cc}(T)=\mathrm{cc}(\mathrm{jdt}(T)).$ Consider the bijection (\ref{nF 2.2}) with $a\geq \lambda'_s$. Under the identification $\mathrm{Tab}(\xi_{\Bla,a},\mu'')$ with $\mathrm{Tab}(\Bla,\mu''),$ we have $\mathrm{cc}(T)=\mathrm{cc}(S)$ if $\Gamma_a(T)=(D,S)$ (note that $S=\mathrm{jdt}(T)$ by Theorem \ref{thm 1.6.3}). We have
\begin{eqnarray*}
M(\mu'' ,\boldsymbol{\lambda} ;t) &= &\sum\limits_{(\nu ,J)\in \mathrm{RC}(\mu'',\boldsymbol{\lambda})}t^{\mathrm{cc}(\nu ,J)}\\
                                &= &\sum\limits_{T\in \mathrm{SST}(\boldsymbol{\lambda},\mu'')}t^{\mathrm{cc}(T)}\\
                                &= &\sum\limits_{%
                                \eta\in\mathscr{P}_{N}}c^{\eta}_{\lambda',\lambda''}\left(\sum\limits_{S\in\mathrm{SST}(\eta,\mu'')}t^{\mathrm{cc}(S)}\right)\\
                                &= &\sum\limits_{%
                                \nu\in\mathscr{P}_{N}}c^{\eta}_{\lambda',\lambda''}\left(\sum\limits_{(\nu,J)\in\mathrm{RC}(\mu'',\eta)}t^{\mathrm{cc}(\nu,J)}\right)\\
                                &= &\sum\limits_{%
                                \nu\in\mathscr{P}_{N}}c^{\eta}_{\lambda',\lambda''}M(\eta ,\mu'' ;t)\\
                                &= & t^{n(\mu)}\sum\limits_{%
                                \eta\in\mathscr{P}_{N}}c^{\eta}_{\lambda',\lambda''}K_{\eta ,\mu''}(t^{-1}).
\end{eqnarray*}
The second equality follows from the fact that the bijection $\mathrm{RC}(\mu'',\Bla)\isom \mathrm{Tab}(\Bla,\mu'')$ preserves the cocharge, the third equality follows from Theorem \ref{thm 1.6.3} as in the proof of Proposition \ref{nthm 2.1.1}. The fourth equality follows as in the second, by applying the bijection $\mathrm{Tab}(\eta,\mu'')\isom\mathrm{RC}(\mu'',\eta).$ Then we obtain the last formula by (\ref{F 4.11}).
Now the theorem follows from Proposition \ref{nthm 2.1.1}.
\end{proof}

\begin{rem}
\emph{For $\mu\in\mathscr{P}_N$ and $\Bla\in\mathscr{P}_{N,2}$ as in 4.7,} \emph{we define a $1D$ sum}
\begin{equation}
X(\mu,\Bla;t)=\sum_{b\in P(B(\mu),\Bla)}t^{\mathrm{E}(b)}.
\end{equation}
\emph{Combining Theorem \ref{thm 4.6.3} with Theorem \ref{thm 3.9.3}, we obtain the following formula, which is an analogue of the $X=M$ conjecture. }
\begin{equation}
X(\mu,\Bla;t)=t^{n(\mu)}M(\mu,\Bla;t^{-1}).
\end{equation}
\end{rem}

\begin{center}
{\sc Appendix  \  Examples for Proposition 4.9.2 }
\end{center}

\par\medskip
We give some examples for Proposition 4.9.2. We will depict a rigged configuration as a sequence of Young diagrams with rows labeled on the left by the corresponding vacancy number and labeled on the right by the corresonding rigging. For example,

\begin{center}
\unitlength 10pt
\begin{picture}(20,5)

\Yboxdim{10pt}
\put(1,1){\yng(3,1)}
\put(0.2,1){1}
\put(0.2,2){1}
\put(2.2,1){0}
\put(4.2,2){0}
\put(7,1){\yng(3,1)}
\put(6.2,1){1}
\put(6.2,2){1}
\put(8.2,1){1}
\put(10.2,2){0}
\put(13,2){\yng(1)}
\put(12.2,2){0}
\put(14.2,2){0}
\end{picture}
\end{center}

(1) Let $\Bla=((4,2),(4,2,2,1))$ and $\mu=(4,4,3,2,2).$ Let
\par\noindent  $T=(\begin{array}{llll}
1 & 2 & 2 &3\\
3 & 5
\end{array},\begin{array}{llll}
1 & 1 & 1 & 2\\
2 & 4  \\
3 & 5  \\
4
\end{array})\in \mathrm{Tab}(\Bla,\mu).$ The corresponding rigged configuration $(\nu,J)$ is computed in the following table.

\medskip
\noindent
\begin{tabu}to \hsize {|X[0.2,1]|X[4,1]|X[7,1]|}
\hline
  $i$ & $\BBw_i$                     & $(\nu,J)_i$                          \\
  \hline
  0   & $(-,-,-,-,-,-)$              & $\emptyset $                         \\
  1   & $(-,-,-,1,-,-)$              & $\Yvcentermath1 \put(5,-3){\Yboxdim{8pt}\yng(1)}\put(-1,-2){\tiny{0}} \put(14,-2){\tiny{0}}   \put(62,-3){\Yboxdim{8pt}\yng(1)}\put(53,-2){\tiny{-1}}\put(72,-2){\tiny{-1}}\smallskip $               \\                                                                                                                                                 2   & $(-,-,-,11,-,-)$             & $\Yvcentermath1 \put(5,-3){\Yboxdim{8pt}\yng(2)}\put(-1,-2){\tiny{0}} \put(21,-2){\tiny{0}}   \put(62,-3){\Yboxdim{8pt}\yng(2)}\put(53,-2){\tiny{-2}}\put(79,-2){\tiny{-2}}\smallskip $                              \\
  3   & $(-,-,-,111,-,-)$            & $\Yvcentermath1 \put(5,-3){\Yboxdim{8pt}\yng(3)}\put(-1,-2){\tiny{0}} \put(29,-2){\tiny{0}}   \put(62,-3){\Yboxdim{8pt}\yng(3)}\put(53,-2){\tiny{-3}}\put(87,-2){\tiny{-3}}\smallskip $                                \\
  4   & $(-,-,-,111,-,1)$            & $\Yvcentermath1 \put(5,-3){\Yboxdim{8pt}\yng(3)}\put(-1,-2){\tiny{0}} \put(29,-2){\tiny{0}}   \put(62,-3){\Yboxdim{8pt}\yng(3)}\put(53,-2){\tiny{-3}}\put(87,-2){\tiny{-3}}\smallskip $                                \\
  5   & $(-,-,2,111,-,1)$            & $\Yvcentermath1 \put(5,-3){\Yboxdim{8pt}\yng(3,1)}\put(-1,6){\tiny{0}} \put(29,6){\tiny{0}} \put(-1,-2){\tiny{0}} \put(14,-2){\tiny{0}}   \put(62,-3){\Yboxdim{8pt}\yng(3,1)}\put(54,6){\tiny{-3}}\put(86,6){\tiny{-3}}   \put(53,-2){\tiny{-1}}\put(71,-2){\tiny{-1}} \put(116,1){\Yboxdim{8pt}\yng(1)} \put(111,1){\tiny{0}} \put(125,1){\tiny{0}}  \smallskip $                              \\
  6   & $(-,-,2,1112,-,1)$           & $\Yvcentermath1 \put(5,-3){\Yboxdim{8pt}\yng(4,1)}\put(-1,6){\tiny{1}} \put(37,6){\tiny{1}} \put(-1,-2){\tiny{0}} \put(14,-2){\tiny{0}}   \put(62,-3){\Yboxdim{8pt}\yng(4,1)}\put(54,6){\tiny{-4}}\put(94,6){\tiny{-4}}   \put(53,-2){\tiny{-1}}\put(71,-2){\tiny{-1}} \put(116,1){\Yboxdim{8pt}\yng(1)} \put(111,1){\tiny{0}} \put(125,1){\tiny{0}}  \smallskip $                                 \\
  7   & $(-,-,2,1112,-,12)$          & $\Yvcentermath1 \put(5,-3){\Yboxdim{8pt}\yng(4,1)}\put(-1,6){\tiny{2}} \put(37,6){\tiny{1}} \put(-1,-2){\tiny{0}} \put(14,-2){\tiny{0}}   \put(62,-3){\Yboxdim{8pt}\yng(4,1)}\put(54,6){\tiny{-4}}\put(94,6){\tiny{-4}}   \put(53,-2){\tiny{-1}}\put(71,-2){\tiny{-1}} \put(116,1){\Yboxdim{8pt}\yng(1)} \put(111,1){\tiny{0}} \put(125,1){\tiny{0}}  \smallskip$                              \\
  8   & $(-,-,2,1112,-,122)$         & $\Yvcentermath1 \put(5,-3){\Yboxdim{8pt}\yng(4,1)}\put(-1,6){\tiny{3}} \put(37,6){\tiny{1}} \put(-1,-2){\tiny{0}} \put(14,-2){\tiny{0}}   \put(62,-3){\Yboxdim{8pt}\yng(4,1)}\put(54,6){\tiny{-4}}\put(94,6){\tiny{-4}}   \put(53,-2){\tiny{-1}}\put(71,-2){\tiny{-1}} \put(116,1){\Yboxdim{8pt}\yng(1)} \put(111,1){\tiny{0}} \put(125,1){\tiny{0}}  \smallskip$                            \\
  9   & $(-,3,2,1112,-,122)$         & $\Yvcentermath1 \put(5,-3){\Yboxdim{8pt}\yng(4,1,1)}\put(-1,14){\tiny{3}} \put(37,14){\tiny{1}} \put(-1,-2){\tiny{0}} \put(14,-2){\tiny{0}} \put(-1,6){\tiny{0}} \put(14,6){\tiny{0}}   \put(62,-3){\Yboxdim{8pt}\yng(4,1,1)}\put(54,14){\tiny{-4}}\put(94,14){\tiny{-4}}   \put(53,-2){\tiny{-1}}\put(71,-2){\tiny{-1}} \put(53,6){\tiny{-1}}\put(71,6){\tiny{-1}} \put(116,1){\Yboxdim{8pt}\yng(1,1)} \put(111,1){\tiny{0}} \put(125,1){\tiny{0}}\put(111,9){\tiny{0}} \put(125,9){\tiny{0}}  \put(155,5){\Yboxdim{8pt}\yng(1)}  \put(150,5){\tiny{0}} \put(164,5){\tiny{0}}\smallskip$                        \\
  10  & $(-,3,2,1112,3,122)$         & $\Yvcentermath1 \put(5,-3){\Yboxdim{8pt}\yng(4,2,1)}\put(-1,14){\tiny{2}} \put(37,14){\tiny{1}} \put(-1,-2){\tiny{0}} \put(14,-2){\tiny{0}} \put(-1,6){\tiny{0}} \put(22,6){\tiny{0}}   \put(62,-3){\Yboxdim{8pt}\yng(4,1,1)}\put(54,14){\tiny{-3}}\put(94,14){\tiny{-4}}   \put(53,-2){\tiny{-1}}\put(71,-2){\tiny{-1}} \put(53,6){\tiny{-1}}\put(71,6){\tiny{-1}} \put(116,1){\Yboxdim{8pt}\yng(1,1)} \put(111,1){\tiny{0}} \put(125,1){\tiny{0}}\put(111,9){\tiny{0}} \put(125,9){\tiny{0}}  \put(155,5){\Yboxdim{8pt}\yng(1)}  \put(150,5){\tiny{0}} \put(164,5){\tiny{0}}\smallskip$                          \\
  11  & $(-,3,2,1112,3,1223)$        & $\Yvcentermath1 \put(5,-3){\Yboxdim{8pt}\yng(4,2,1)}\put(-1,14){\tiny{3}} \put(37,14){\tiny{1}} \put(-1,-2){\tiny{0}} \put(14,-2){\tiny{0}} \put(-1,6){\tiny{0}} \put(22,6){\tiny{0}}   \put(62,-3){\Yboxdim{8pt}\yng(4,1,1)}\put(54,14){\tiny{-3}}\put(94,14){\tiny{-4}}   \put(53,-2){\tiny{-1}}\put(71,-2){\tiny{-1}} \put(53,6){\tiny{-1}}\put(71,6){\tiny{-1}} \put(116,1){\Yboxdim{8pt}\yng(1,1)} \put(111,1){\tiny{0}} \put(125,1){\tiny{0}}\put(111,9){\tiny{0}} \put(125,9){\tiny{0}}  \put(155,5){\Yboxdim{8pt}\yng(1)}  \put(150,5){\tiny{0}} \put(164,5){\tiny{0}}\smallskip$                                \\
  12  & $(4,3,2,1112,3,1223)$        & $\Yvcentermath1 \put(5,-3){\Yboxdim{8pt}\yng(4,2,1,1)}\put(-1,22){\tiny{3}} \put(37,22){\tiny{1}} \put(-1,13){\tiny{0}} \put(22,13){\tiny{0}} \put(-1,6){\tiny{0}} \put(14,6){\tiny{0}} \put(-1,-2){\tiny{0}} \put(14,-2){\tiny{0}}   \put(62,-3){\Yboxdim{8pt}\yng(4,1,1,1)}\put(54,22){\tiny{-3}}\put(94,22){\tiny{-4}}   \put(53,14){\tiny{-1}}\put(71,14){\tiny{-1}} \put(53,6){\tiny{-1}}\put(71,6){\tiny{-1}} \put(53,-2){\tiny{-1}}\put(71,-2){\tiny{-1}} \put(116,1){\Yboxdim{8pt}\yng(1,1,1)} \put(111,17){\tiny{0}} \put(125,17){\tiny{0}}\put(111,9){\tiny{0}} \put(125,9){\tiny{0}} \put(111,1){\tiny{0}} \put(125,1){\tiny{0}} \put(155,5){\Yboxdim{8pt}\yng(1,1)}  \put(150,13){\tiny{0}} \put(164,13){\tiny{0}} \put(150,5){\tiny{0}} \put(164,5){\tiny{0}} \put(194,9){\Yboxdim{8pt}\yng(1)}  \put(189,9){\tiny{0}} \put(203,9){\tiny{0}}\smallskip$                              \\
  13  & $(4,3,24,1112,3,1223)$       & $\Yvcentermath1 \put(5,-3){\Yboxdim{8pt}\yng(4,2,2,1)}\put(-1,22){\tiny{3}} \put(37,22){\tiny{1}} \put(-1,13){\tiny{0}} \put(22,13){\tiny{0}} \put(-1,6){\tiny{0}} \put(22,6){\tiny{0}} \put(-1,-2){\tiny{0}} \put(14,-2){\tiny{0}}   \put(62,-3){\Yboxdim{8pt}\yng(4,2,1,1)}\put(54,22){\tiny{-3}}\put(94,22){\tiny{-4}}   \put(53,14){\tiny{-1}}\put(79,14){\tiny{-1}} \put(53,6){\tiny{-1}}\put(71,6){\tiny{-1}} \put(53,-2){\tiny{-1}}\put(71,-2){\tiny{-1}} \put(116,1){\Yboxdim{8pt}\yng(2,1,1)} \put(111,17){\tiny{0}} \put(133,17){\tiny{0}}\put(111,9){\tiny{0}} \put(125,9){\tiny{0}} \put(111,1){\tiny{0}} \put(125,1){\tiny{0}} \put(155,5){\Yboxdim{8pt}\yng(1,1)}  \put(150,13){\tiny{0}} \put(164,13){\tiny{0}} \put(150,5){\tiny{0}} \put(164,5){\tiny{0}} \put(194,9){\Yboxdim{8pt}\yng(1)}  \put(189,9){\tiny{0}} \put(203,9){\tiny{0}}\smallskip$                                 \\
  14  & $(4,35,24,1112,3,1223)$      & $\Yvcentermath1 \put(5,-3){\Yboxdim{8pt}\yng(4,2,2,2)}\put(-1,22){\tiny{3}} \put(37,22){\tiny{1}} \put(-1,13){\tiny{0}} \put(22,13){\tiny{0}} \put(-1,6){\tiny{0}} \put(22,6){\tiny{0}} \put(-1,-2){\tiny{0}} \put(22,-2){\tiny{0}}   \put(62,-3){\Yboxdim{8pt}\yng(4,2,2,1)}\put(54,22){\tiny{-3}}\put(94,22){\tiny{-4}}   \put(53,14){\tiny{-1}}\put(79,14){\tiny{-1}} \put(53,6){\tiny{-1}}\put(79,6){\tiny{-1}} \put(53,-2){\tiny{-1}}\put(71,-2){\tiny{-1}} \put(116,1){\Yboxdim{8pt}\yng(2,2,1)} \put(111,17){\tiny{0}} \put(133,17){\tiny{0}}\put(111,9){\tiny{0}} \put(133,9){\tiny{0}} \put(111,1){\tiny{0}} \put(125,1){\tiny{0}} \put(155,5){\Yboxdim{8pt}\yng(2,1)}  \put(150,13){\tiny{0}} \put(172,13){\tiny{0}} \put(150,5){\tiny{0}} \put(164,5){\tiny{0}} \put(194,9){\Yboxdim{8pt}\yng(1)}  \put(189,9){\tiny{0}} \put(203,9){\tiny{0}}\smallskip$                              \\
  15  & $(4,35,24,1112,35,1223)$     & $\Yvcentermath1 \put(5,-3){\Yboxdim{8pt}\yng(4,3,2,2)}\put(-1,22){\tiny{2}} \put(37,22){\tiny{1}} \put(-1,13){\tiny{1}} \put(30,13){\tiny{1}} \put(-1,6){\tiny{1}} \put(22,6){\tiny{0}} \put(-1,-2){\tiny{1}} \put(22,-2){\tiny{0}}   \put(62,-3){\Yboxdim{8pt}\yng(4,2,2,1)}\put(54,22){\tiny{-2}}\put(94,22){\tiny{-4}}   \put(53,14){\tiny{-1}}\put(79,14){\tiny{-1}} \put(53,6){\tiny{-1}}\put(79,6){\tiny{-1}} \put(53,-2){\tiny{-1}}\put(71,-2){\tiny{-1}} \put(116,1){\Yboxdim{8pt}\yng(2,2,1)} \put(111,17){\tiny{0}} \put(133,17){\tiny{0}}\put(111,9){\tiny{0}} \put(133,9){\tiny{0}} \put(111,1){\tiny{0}} \put(125,1){\tiny{0}} \put(155,5){\Yboxdim{8pt}\yng(2,1)}  \put(150,13){\tiny{0}} \put(172,13){\tiny{0}} \put(150,5){\tiny{0}} \put(164,5){\tiny{0}} \put(194,9){\Yboxdim{8pt}\yng(1)}  \put(189,9){\tiny{0}} \put(203,9){\tiny{0}}\smallskip$                                \\
   \hline
\end{tabu}

\medskip

\noindent We obtain $(\nu,J)=\Yvcentermath1 \put(5,-17){\Yboxdim{10pt}\yng(4,3,2,2)} \put(45,14){\small{1}} \put(35,4){\small{1}} \put(25,-6){\small{0}} \put(25,-16){\small{0}} \put(82,-17){\Yboxdim{10pt}\yng(4,2,2,1)}\put(122,14){\small{-4}} \put(102,4){\small{-1}} \put(102,-6){\small{-1}} \put(92,-16){\small{-1}} \put(156,-13){\Yboxdim{10pt}\yng(2,2,1)}\put(176,8){\small{0}} \put(176,-2){\small{0}} \put(166,-12){\small{0}}\put(215,-9){\Yboxdim{10pt}\yng(2,1)} \put(235,1){\small{0}} \put(225,-9){\small{0}} \put(274,-3){\Yboxdim{10pt}\yng(1)} \put(284,-3){\small{0}} $

\medskip

\noindent Hence, we have $(\nu,J_{+})=\Yvcentermath1 \put(5,-17){\Yboxdim{10pt}\yng(4,3,2,2)} \put(45,14){\small{1}} \put(35,4){\small{1}} \put(25,-6){\small{0}} \put(25,-16){\small{0}} \put(82,-17){\Yboxdim{10pt}\yng(4,2,2,1)}\put(122,14){\small{0}} \put(102,4){\small{1}} \put(102,-6){\small{1}} \put(92,-16){\small{0}} \put(156,-13){\Yboxdim{10pt}\yng(2,2,1)}\put(176,8){\small{0}} \put(176,-2){\small{0}} \put(166,-12){\small{0}}\put(215,-9){\Yboxdim{10pt}\yng(2,1)} \put(235,1){\small{0}} \put(225,-9){\small{0}} \put(274,-3){\Yboxdim{10pt}\yng(1)} \put(284,-3){\small{0}} $

\medskip

\noindent One can easily check that $(\nu,J_+)\in\mathrm{QM}(\mu,\boldsymbol{\lambda}).$

\bigskip

(2) Let $\Bla=((4,2),(3,3,2,1))$ and $\mu=(4,3,3,2,2,1).$ Let
\par\noindent $T=(\begin{array}{llll}
1 & 1 & 2 &4\\
2 & 3
\end{array},\begin{array}{llll}
1 & 1 & 2\\
3 & 5 & 7\\
4 & 6\\
5
\end{array})\in \mathrm{Tab}(\Bla,\mu).$ The corresponding rigged configuration $(\nu,J)$ is computed in the following table.

\medskip
\noindent
\begin{tabu}to \hsize {|X[0.2,1]|X[4,1]|X[7,1]|}
\hline
    $i$ & $\BBw_i$                     & $(\nu,J)_i$                          \\
\hline
0       &(-,-,-,-,-,-)                 & $\emptyset $  \\
1       &(-,-,-,1,-,-)                 & $\Yvcentermath1 \put(5,-3){\Yboxdim{8pt}\yng(1)}\put(-1,-2){\tiny{0}} \put(14,-2){\tiny{0}}   \put(62,-3){\Yboxdim{8pt}\yng(1)}\put(53,-2){\tiny{-1}}\put(72,-2){\tiny{-1}}\smallskip $   \\
2       &(-,-,-,11,-,-)                & $\Yvcentermath1 \put(5,-3){\Yboxdim{8pt}\yng(2)}\put(-1,-2){\tiny{0}} \put(21,-2){\tiny{0}}   \put(62,-3){\Yboxdim{8pt}\yng(2)}\put(53,-2){\tiny{-2}}\put(79,-2){\tiny{-2}}\smallskip $   \\
3       &(-,-,-,11,-,1)                & $\Yvcentermath1 \put(5,-3){\Yboxdim{8pt}\yng(2)}\put(-1,-2){\tiny{0}} \put(21,-2){\tiny{0}}   \put(62,-3){\Yboxdim{8pt}\yng(2)}\put(53,-2){\tiny{-2}}\put(79,-2){\tiny{-2}}\smallskip $  \\
4       &(-,-,-,11,-,11)               & $\Yvcentermath1 \put(5,-3){\Yboxdim{8pt}\yng(2)}\put(-1,-2){\tiny{0}} \put(21,-2){\tiny{0}}   \put(62,-3){\Yboxdim{8pt}\yng(2)}\put(53,-2){\tiny{-2}}\put(79,-2){\tiny{-2}}\smallskip $   \\
5       &(-,-,-,112,-,11)              & $\Yvcentermath1 \put(5,-3){\Yboxdim{8pt}\yng(3)}\put(-1,-2){\tiny{1}} \put(29,-2){\tiny{1}}   \put(62,-3){\Yboxdim{8pt}\yng(3)}\put(53,-2){\tiny{-3}}\put(87,-2){\tiny{-3}}\smallskip $   \\
6       &(-,-,-,112,2,11)              & $\Yvcentermath1 \put(5,-3){\Yboxdim{8pt}\yng(4)}\put(-1,-2){\tiny{1}} \put(37,-2){\tiny{1}}   \put(62,-3){\Yboxdim{8pt}\yng(3)}\put(53,-2){\tiny{-3}}\put(87,-2){\tiny{-3}}\smallskip $  \\
7       &(-,-,-,112,2,112)             & $\Yvcentermath1 \put(5,-3){\Yboxdim{8pt}\yng(4)}\put(-1,-2){\tiny{2}} \put(37,-2){\tiny{1}}   \put(62,-3){\Yboxdim{8pt}\yng(3)}\put(53,-2){\tiny{-3}}\put(87,-2){\tiny{-3}}\smallskip $  \\
8       &(-,-,3,112,2,112)             & $\Yvcentermath1 \put(5,-7){\Yboxdim{8pt}\yng(4,1)}\put(-1,2){\tiny{2}} \put(37,2){\tiny{1}} \put(-1,-6){\tiny{1}} \put(14,-6){\tiny{1}}  \put(62,-7){\Yboxdim{8pt}\yng(3,1)}\put(53,2){\tiny{-3}}\put(87,2){\tiny{-3}}\put(53,-6){\tiny{-1}}\put(71,-6){\tiny{-1}} \put(117,-3){\Yboxdim{8pt}\yng(1)}\put(112,-3){\tiny{0}} \put(126,-3){\tiny{0}} \smallskip $  \\
9       &(-,-,3,112,23,112)            & $\Yvcentermath1 \put(5,-7){\Yboxdim{8pt}\yng(4,2)}\put(-1,2){\tiny{1}} \put(37,2){\tiny{1}} \put(-1,-6){\tiny{1}} \put(22,-6){\tiny{1}}  \put(62,-7){\Yboxdim{8pt}\yng(3,1)}\put(53,2){\tiny{-2}}\put(87,2){\tiny{-3}}\put(53,-6){\tiny{-1}}\put(71,-6){\tiny{-1}} \put(117,-3){\Yboxdim{8pt}\yng(1)}\put(112,-3){\tiny{0}} \put(126,-3){\tiny{0}} \smallskip $  \\
10      &(-,4,3,112,23,112)            & $\Yvcentermath1 \put(5,-7){\Yboxdim{8pt}\yng(4,2,1)}\put(-1,10){\tiny{1}} \put(37,10){\tiny{1}} \put(-1,2){\tiny{1}} \put(22,2){\tiny{1}}  \put(-1,-6){\tiny{1}} \put(14,-6){\tiny{1}} \put(62,-7){\Yboxdim{8pt}\yng(3,1,1)}\put(53,10){\tiny{-2}}\put(87,10){\tiny{-3}}\put(53,2){\tiny{-1}}\put(71,2){\tiny{-1}} \put(53,-6){\tiny{-1}}\put(71,-6){\tiny{-1}} \put(117,-3){\Yboxdim{8pt}\yng(1,1)}\put(112,5){\tiny{0}} \put(126,5){\tiny{0}} \put(112,-3){\tiny{0}} \put(126,-3){\tiny{0}} \put(172,1){\Yboxdim{8pt}\yng(1)}\put(167,1){\tiny{0}} \put(181,1){\tiny{0}}\smallskip $  \\
11      &(-,4,3,112,23,1124)           & $\Yvcentermath1 \put(5,-7){\Yboxdim{8pt}\yng(4,2,1)}\put(-1,10){\tiny{2}} \put(37,10){\tiny{1}} \put(-1,2){\tiny{2}} \put(22,2){\tiny{1}}  \put(-1,-6){\tiny{1}} \put(14,-6){\tiny{1}} \put(62,-7){\Yboxdim{8pt}\yng(3,1,1)}\put(53,10){\tiny{-2}}\put(87,10){\tiny{-3}}\put(53,2){\tiny{-1}}\put(71,2){\tiny{-1}} \put(53,-6){\tiny{-1}}\put(71,-6){\tiny{-1}}
\put(117,-3){\Yboxdim{8pt}\yng(1,1)}\put(112,5){\tiny{0}} \put(126,5){\tiny{0}} \put(112,-3){\tiny{0}} \put(126,-3){\tiny{0}} \put(172,1){\Yboxdim{8pt}\yng(1)}\put(167,1){\tiny{0}} \put(181,1){\tiny{0}}\smallskip $  \\
12      &(5,4,3,112,23,1124)           & $\Yvcentermath1 \put(5,-7){\Yboxdim{8pt}\yng(4,2,1,1)}\put(-1,18){\tiny{2}} \put(37,18){\tiny{1}} \put(-1,10){\tiny{2}} \put(22,10){\tiny{1}}  \put(-1,2){\tiny{1}} \put(14,2){\tiny{1}} \put(-1,-6){\tiny{1}} \put(14,-6){\tiny{1}} \put(62,-7){\Yboxdim{8pt}\yng(3,1,1,1)}\put(53,18){\tiny{-2}}\put(87,18){\tiny{-3}}\put(53,10){\tiny{-1}}\put(71,10){\tiny{-1}} \put(53,2){\tiny{-1}}\put(71,2){\tiny{-1}}\put(53,-6){\tiny{-1}}\put(71,-6){\tiny{-1}}
\put(117,-3){\Yboxdim{8pt}\yng(1,1,1)}\put(112,13){\tiny{0}} \put(126,13){\tiny{0}} \put(112,5){\tiny{0}} \put(126,5){\tiny{0}} \put(112,-3){\tiny{0}} \put(126,-3){\tiny{0}}
\put(172,1){\Yboxdim{8pt}\yng(1,1)}\put(167,9){\tiny{0}} \put(181,9){\tiny{0}}\put(167,1){\tiny{0}} \put(181,1){\tiny{0}}
\put(217,5){\Yboxdim{8pt}\yng(1)}\put(212,5){\tiny{0}} \put(226,5){\tiny{0}}\smallskip $  \\
13      &(5,4,35,112,23,1124)          & $\Yvcentermath1 \put(5,-7){\Yboxdim{8pt}\yng(4,2,2,1)}\put(-1,18){\tiny{2}} \put(37,18){\tiny{1}} \put(-1,10){\tiny{2}} \put(22,10){\tiny{2}}
\put(-1,2){\tiny{2}} \put(22,2){\tiny{1}}
\put(-1,-6){\tiny{1}} \put(14,-6){\tiny{1}}
\put(62,-7){\Yboxdim{8pt}\yng(3,2,1,1)}\put(53,18){\tiny{-2}}\put(87,18){\tiny{-3}}
\put(53,10){\tiny{-1}}\put(79,10){\tiny{-1}}
\put(53,2){\tiny{-1}}\put(71,2){\tiny{-1}}
\put(53,-6){\tiny{-1}}\put(71,-6){\tiny{-1}}
\put(117,-3){\Yboxdim{8pt}\yng(2,1,1)}\put(112,13){\tiny{0}} \put(134,13){\tiny{0}}
\put(112,5){\tiny{0}} \put(126,5){\tiny{0}}
\put(112,-3){\tiny{0}} \put(126,-3){\tiny{0}}
\put(172,1){\Yboxdim{8pt}\yng(1,1)}\put(167,9){\tiny{0}} \put(181,9){\tiny{0}}
\put(167,1){\tiny{0}} \put(181,1){\tiny{0}}
\put(217,5){\Yboxdim{8pt}\yng(1)}\put(212,5){\tiny{0}} \put(226,5){\tiny{0}}\smallskip $  \\
14      &(5,46,35,112,23,1124)         &  $\Yvcentermath1 \put(5,-7){\Yboxdim{8pt}\yng(4,2,2,2)}\put(-1,18){\tiny{2}} \put(37,18){\tiny{1}} \put(-1,10){\tiny{2}} \put(22,10){\tiny{2}}
\put(-1,2){\tiny{2}} \put(22,2){\tiny{2}}
\put(-1,-6){\tiny{2}} \put(22,-6){\tiny{1}}
\put(62,-7){\Yboxdim{8pt}\yng(3,2,2,1)}\put(53,18){\tiny{-2}}\put(87,18){\tiny{-3}}
\put(53,10){\tiny{-1}}\put(79,10){\tiny{-1}}
\put(53,2){\tiny{-1}}\put(79,2){\tiny{-1}}
\put(53,-6){\tiny{-1}}\put(71,-6){\tiny{-1}}
\put(117,-3){\Yboxdim{8pt}\yng(2,2,1)}\put(112,13){\tiny{0}} \put(134,13){\tiny{0}}
\put(112,5){\tiny{0}} \put(134,5){\tiny{0}}
\put(112,-3){\tiny{0}} \put(126,-3){\tiny{0}}
\put(172,1){\Yboxdim{8pt}\yng(2,1)}\put(167,9){\tiny{0}} \put(189,9){\tiny{0}}
\put(167,1){\tiny{0}} \put(181,1){\tiny{0}}
\put(217,5){\Yboxdim{8pt}\yng(1)}\put(212,5){\tiny{0}} \put(226,5){\tiny{0}}\smallskip $ \\
15      &(5,46,357,112,23,1124)        & $\Yvcentermath1 \put(5,-7){\Yboxdim{8pt}\yng(4,3,2,2)}\put(-1,18){\tiny{2}} \put(37,18){\tiny{1}} \put(-1,10){\tiny{3}} \put(30,10){\tiny{3}}
\put(-1,2){\tiny{3}} \put(22,2){\tiny{2}}
\put(-1,-6){\tiny{3}} \put(22,-6){\tiny{1}}
\put(62,-7){\Yboxdim{8pt}\yng(3,3,2,1)}\put(53,18){\tiny{-2}}\put(87,18){\tiny{-2}}
\put(53,10){\tiny{-2}}\put(87,10){\tiny{-3}}
\put(53,2){\tiny{-1}}\put(79,2){\tiny{-1}}
\put(53,-6){\tiny{-1}}\put(71,-6){\tiny{-1}}
\put(117,-3){\Yboxdim{8pt}\yng(3,2,1)}\put(112,13){\tiny{0}} \put(142,13){\tiny{0}}
\put(112,5){\tiny{0}} \put(134,5){\tiny{0}}
\put(112,-3){\tiny{0}} \put(126,-3){\tiny{0}}
\put(172,1){\Yboxdim{8pt}\yng(2,1)}\put(167,9){\tiny{0}} \put(189,9){\tiny{0}}
\put(167,1){\tiny{0}} \put(181,1){\tiny{0}}
\put(217,5){\Yboxdim{8pt}\yng(1)}\put(212,5){\tiny{0}} \put(226,5){\tiny{0}}\smallskip $  \\
 \hline
\end{tabu}

\medskip

\noindent We obtain $(\nu,J)=\Yvcentermath1 \put(5,-17){\Yboxdim{10pt}\yng(4,3,2,2)} \put(45,14){\small{1}} \put(35,4){\small{3}} \put(25,-6){\small{2}} \put(25,-16){\small{1}}
\put(82,-17){\Yboxdim{10pt}\yng(3,3,2,1)}\put(112,14){\small{-2}} \put(112,4){\small{-3}} \put(102,-6){\small{-1}} \put(92,-16){\small{-1}} \put(156,-13){\Yboxdim{10pt}\yng(3,2,1)}\put(186,8){\small{0}} \put(176,-2){\small{0}} \put(166,-12){\small{0}}\put(215,-9){\Yboxdim{10pt}\yng(2,1)} \put(235,1){\small{0}} \put(225,-9){\small{0}} \put(274,-3){\Yboxdim{10pt}\yng(1)} \put(284,-3){\small{0}} $

\medskip

\noindent Hence, we have $(\nu,J_{+})=\Yvcentermath1 \put(5,-17){\Yboxdim{10pt}\yng(4,3,2,2)} \put(45,14){\small{1}} \put(35,4){\small{3}} \put(25,-6){\small{2}} \put(25,-16){\small{1}}
\put(82,-17){\Yboxdim{10pt}\yng(3,3,2,1)}\put(112,14){\small{1}} \put(112,4){\small{0}} \put(102,-6){\small{1}} \put(92,-16){\small{0}} \put(156,-13){\Yboxdim{10pt}\yng(3,2,1)}\put(186,8){\small{0}} \put(176,-2){\small{0}} \put(166,-12){\small{0}}\put(215,-9){\Yboxdim{10pt}\yng(2,1)} \put(235,1){\small{0}} \put(225,-9){\small{0}} \put(274,-3){\Yboxdim{10pt}\yng(1)} \put(284,-3){\small{0}} $

\medskip

\noindent One can easily check that $(\nu,J_+)\in\mathrm{QM}(\mu,\boldsymbol{\lambda}).$

\bigskip

(3) Let $\Bla=((2,1),(2))$ and $\mu=(2,2,1).$ We list all  the elements of $\mathrm{Tab}(\Bla,\mu)$ and the corresponding rigged configurations in the following table.

\medskip
\noindent
\begin{tabu}to \hsize {|X[4,1]|X[7,1]|}
\hline
    $T$                     &         $(\nu,J)$                          \\
\hline
    $(\begin{array}{ll}
1 & 1\\
2
\end{array}, \begin{array}{ll}
2 & 3
\end{array} )$              &   $\put(60,-11){\yng(2,1)}\put(52,-10){0}\put(80,-10){0}\put(52,4){1}\put(88,4){1}
\put(135,-3){\yng(2)}\put(125,-2){-1}\put(165,-2){-1}\smallskip$                                         \\
$(\begin{array}{ll}
1 & 2\\
2
\end{array}, \begin{array}{ll}
1 & 3
\end{array} )$              &    $\put(60,-8){\yng(2,1)}\put(52,-7){0}\put(80,-7){0}\put(52,7){1}\put(88,7){0}
\put(135,-3){\yng(2)}\put(125,-2){-1}\put(165,-2){-1}\smallskip$                                          \\
$(\begin{array}{ll}
1 & 1\\
3
\end{array}, \begin{array}{ll}
2 & 2
\end{array} )$              &    $\put(60,-3){\yng(3)}\put(52,-2){1} \put(101,-2){1}  \put(135,-3){\yng(2)}\put(125,-2){-2}\put(163,-2){-2}\smallskip$\\
$(\begin{array}{ll}
1 & 2\\
3
\end{array}, \begin{array}{ll}
1 & 2
\end{array} )$              &    $\put(60,-8){\yng(2,1)}\put(52,-7){0}\put(80,-7){0}\put(52,7){1}\put(88,7){1}
\put(135,-3){\yng(2)}\put(125,-2){-2}\put(165,-2){-2}\smallskip$\\
$(\begin{array}{ll}
2 & 2\\
3
\end{array}, \begin{array}{ll}
1 & 1
\end{array} )$              &   $\put(60,-8){\yng(2,1)}\put(52,-7){0}\put(80,-7){0}\put(52,7){1}\put(88,7){0}
\put(135,-3){\yng(2)}\put(125,-2){-2}\put(165,-2){-2}\smallskip$\\
$(\begin{array}{ll}
1 & 3\\
2
\end{array}, \begin{array}{ll}
1 & 2
\end{array} )$              &    $\put(60,-3){\yng(3)}\put(52,-2){1} \put(101,-2){0}  \put(135,-3){\yng(2)}\put(125,-2){-2}\put(163,-2){-2}\smallskip$\\
\hline
\end{tabu}
\medskip

\noindent One can easily check that $\mathrm{QM}(\mu,\boldsymbol{\lambda})$  consists of $(\nu,J_{+})$ for the $(\nu,J)$ listed in the above table.


\end{document}